\documentclass[a4paper,12pt]{article}

\usepackage{amsthm,amsmath,stmaryrd,bbm,hyperref,geometry,color}
\usepackage[utf8]{inputenc}
\usepackage[english]{babel}
\usepackage{amsfonts,amssymb}
\usepackage{enumitem}

\newcommand{\po}{\left(}
\newcommand{\pf}{\right)}
\newcommand{\co}{\left[}
\newcommand{\cf}{\right]}
\newcommand{\cco}{\llbracket}
\newcommand{\ccf}{\rrbracket}

\newcommand{\R}{\mathbb R}

\newcommand{\N}{\mathbb N} 
\newcommand{\dd}{\mathrm{d}}
\newcommand{\na}{\nabla}

\newcommand{\Contg}{\mathrm{Cont}_{L}(M,\rho)}
\newcommand{\Contd}{\mathrm{Cont}_R(M,\rho)}
\newcommand{\nv}[1]{#1}
\newcommand{\nvv}[1]{#1}

\newtheorem{thm}{Theorem}
\newtheorem{assu}{Assumption}
\newtheorem{lem}[thm]{Lemma}

\newtheorem{prop}[thm]{Proposition}

\title{Almost sure contraction for diffusions on $\R^d$. Application to generalised Langevin diffusions.}
\author{Pierre Monmarché}

\begin{document}
\maketitle


\begin{abstract}
In the case of diffusions on $\R^d$ with constant diffusion matrix, without assuming reversibility nor hypoellipticity, we prove that the contractivity of the deterministic drift is equivalent to the constant rate contraction of Wasserstein distances   $\mathcal W_p$, $p\in[1,\infty]$.  It also implies concentration inequalities for ergodic means of the process. Such a contractivity property is then established for some non-equilibrium chains of anharmonic oscillators and for some generalised Langevin diffusions when the potential is convex with bounded Hessian and the friction is sufficiently high. This extends previous known results for the usual (kinetic) Langevin diffusion.
\end{abstract}
\section{Introduction}

\subsection{Classical contraction results for reversible diffusion processes}

The overdamped Langevin diffusion is the solution   $(X_t)_{t\geqslant 0}$ in $\R^n$ of
\begin{equation}\label{eq:overdamped}
 \dd X_t \ = \ -\na U(X_t) \dd t  + \sqrt 2 \dd W_t
\end{equation}
with $(W_t)_{t\geqslant 0}$ a standard Brownian motion on $\R^n$ and $U\in\mathcal C^1(\R^n)$. When $U$ is uniformly convex, a few things are known for this process. First, considering two solutions of this equation (driven with the same Brownian motion $W$) with different initial conditions, we see that the distance between them almost surely decays at constant rate. This coupling argument yields the contraction of Wasserstein distances $\mathcal W_p$ for all order $p\in[1,+\infty]$ along the Markov semigroup $(P_t)_{t\geqslant 0}$ associated to the process. Besides, the convexity of $U$ is equivalent to the fact that the process \eqref{eq:overdamped} satisfies a curvature-dimension inequality CD($\rho,\infty$) in the sense of Bakry-Émery \cite{BakryGentilLedoux,BakryEmery}, with a curvature  $\rho>0$. For reversible diffusions, CD($\rho,\infty$) is equivalent to $\Gamma(P_t f) \leqslant e^{-2\rho t} P_t(\Gamma f)$ for all sufficiently smooth $f$, where $\Gamma(f)=1/2L(f^2)-fL f$ is the carr\'e du champ associated to the  generator $L$ of the process, and is also equivalent to $\sqrt{\Gamma(P_t f)} \leqslant e^{-\rho t} P_t(\sqrt{\Gamma f})$, see \cite[Theorem 4.7.2 et Theorem 5.5.2]{BakryGentilLedoux}. The positive curvature also implies for the invariant measure of the process (which here is explicit, with density proportional to $\exp(-U)$) a logarithmic Sobolev inequality. In the case of the overdamped Langevin diffusion, $Lf(x) = -\na U(x)\cdot \na f(x) + \Delta f(x)$ and $\Gamma(f) = |\na f|^2$; the convexity of $U$ is thus equivalent to estimates of the form $|\na P_t f|^q \leqslant e^{-q\rho t}P_t(|\na f|^q)$ for $q\in[1,2]$. According to the work of   Kuwada \cite{Kuwada1,Kuwada2}, these contractions of the gradient along the semigroup are equivalent to the  contractions of the distances $\mathcal W_p$ for $p\in[2,+\infty]$. As a consequence, for \eqref{eq:overdamped}, the contraction of $\mathcal W_2$ along $P_t$ is equivalent to the contraction of $\mathcal W_\infty$, and to the fact that $U$ is convex. This has first been stated, along with other equivalences, in the work \cite{VonRenesseSturm} of Von Renesse and Sturm, who showed that in fact the contraction of all the $\mathcal W_p$ distances for $p\in [1,+\infty]$ are equivalent \nvv{see also \cite[5.6.1]{wang2006functional})}. Here  the  process \eqref{eq:overdamped} is a  diffusion, moreover elliptic and reversible.

\subsection{A first motivating non-reversible example}

One of the starting points of the present article stems from the works \cite{BolleyGuillinMalrieu,Chatterji1,Dalalyan} which, each with more or less sharp quantitative results, establish under some conditions a contraction through the parallel coupling for the Langevin diffusion (sometimes called \emph{underdamped} or \emph{kinetic} Langevin diffusion) 
\begin{equation}\label{eq:Langevin_cinetique}
\left\{\begin{array}{rcl}
\dd X_t & = &  Y_t \dd t \\
\dd Y_t & = & -  \na U(X_t) \dd t - \gamma Y_t \dd t + \sqrt{2\gamma} \dd W_t
\end{array}
\right.
\end{equation}
on $\R^n\times\R^n$, with $\gamma>0$ a friction parameter. The process $(X_t,Y_t)_{t\geqslant 0}$ is still a  diffusion; however it is no more  elliptic (although still hypoelliptic); nor reversible (in the probabilist sense, while it is still reversible in the sense of physics, meaning up to a reflection of the velocity). Then, the picture is already less clear.

One may hope that the convexity of $U$ is sufficient to get a contraction for \eqref{eq:Langevin_cinetique}, in particular since the invariant measure has a density proportional to $\exp(-H(x,y))$ with $H(x,y)=U(x)+|y|^2/2$ (which is uniformly convex if $U$ is), which satisfies  \nvv{a logarithmic Sobolev}  inequality. Nevertheless, in \cite{BolleyGuillinMalrieu,Chatterji1,Dalalyan},  $\na^2 U$ has to be uniformly bounded, and moreover the friction $\gamma$ has to be sufficiently high depending on the lower and upper bounds of $\na^2 U$. Yet, in the Gaussian case where $U(x) = x\cdot S x$ with a real symmetric positive definite matrix $S$, there is always a contraction for the parallel coupling for a suitable Euclidean norm (in view of the discussion in Section~\ref{Sec:courbure_general}, this follows from  \cite{MoiGamma,Arnold}), without any restriction on $\gamma$ nor on the eigenvalues of $S$. Could it be that the restrictions on the results of \cite{BolleyGuillinMalrieu,Chatterji1,Dalalyan} only come from sub-optimal proofs and computations ? \nvv{We will see in Section \ref{Sec:suffit-pas} that it is not the case.}

\subsection{Contraction results for non-reversible diffusion processes: general questions}

In view of the Wasserstein contraction results of \cite{BolleyGuillinMalrieu,Chatterji1,Dalalyan} for the underdamped Langevin process \eqref{eq:Langevin_cinetique},  one can wonder what remains of the relations mentioned in the overdamped case. In the latter, $\Gamma(P_t f)$ and $|\na \nvv{P_t f}|^2$ happened to be  the same object.  On the contrary in the kinetic case $\Gamma(f) = |\na_y f|^2$, in particular the Bakry-Émery curvature is $-\infty$. Besides, notice the following: $\Gamma(f)$ is defined from a Markov generator $L$, i.e. some dynamics. On the contrary, $|\na f|$ is defined from a metric, in the general framework, by
the relation
\[|\na f(x)| \ :=  \ \lim_{d(x,y)\rightarrow 0} \frac{|f(x)-f(y)|}{d(x,y)}\]
for some distance $d$. One of the interests of the Bakry-Émery theory is precisely the possibility to define a metric from a Markov dynamics. However, in the case of the Langevin diffusion,  $\R^{2d}$ is already endowed with the Euclidean metric (and the usual gradient), and it is for the latter that Wasserstein contractions are obtained in \cite{BolleyGuillinMalrieu,Chatterji1,Dalalyan}.

\nvv{While previous works concerning the hypocoercive convergence of degenerate diffusion (in particular the work of Villani \cite{Villani2009}) were concerned with integrated quantities (entropy, $L^2$ norms\dots),}
in \cite{Baudoin}, Baudoin  was the first to adapt the Bakry-Émery theory \nvv{based on  local semigroup interpolations} to the process \eqref{eq:Langevin_cinetique}, with a sub-Riemannian geometry-oriented viewpoint. A \nvv{key-point} is to add to the carré du champ $\Gamma$ a ``vertical carré du champ"  $\Gamma^Z$ that complement\nvv{s} $\Gamma$ in the sense that $\Gamma+\Gamma^Z$ is equivalent to the square of the complete  gradient $|\na f|^2$.  This makes possible the adaptation of many classical arguments of semigroup interpolation in this new framework. Nevertheless, is it really necessary to keep $\Gamma$ and to complete it? Why not, instead, working with a gradient that has nothing to do with the carr\'e du champ ? Here are three examples pushing in this direction: \emph{1)} for a non-diffusive process, the carré du champ is a non-local operator that has in general nothing to do with $|\na f|^2$, yet in some cases it is possible to obtain, through Bakry-Émery type arguments, some estimates on $|\na P_t f|^2$, see \cite{MoiPDMP,MoiGamma}; \emph{2)} for generalised Ornstein-Uhlenbeck processes, i.e. for diffusions with a linear drift and a constant diffusion matrix, we can see in \cite{MoiGamma,Arnold} that, in order to establish an optimal convergence rate, one should work with quantities of the form $|Q\na f|^2$ where the definition of the matrix $Q$ only involves the drift matrix, and thus does not depend on the diffusion matrix at all, while the    carré du champ only depends on the diffusion matrix; \emph{3)} for a diffusion with a constant matrix diffusion, by design, the parallel coupling is blind to the diffusion matrix (from which, again, the   carré du champ  is defined), the contraction only comes from the deterministic drift.  

Besides, this last remark \nvv{indicates that, although hypoellipticity is often associated to convergence to equilibrium, and while the initial Sturm-Von Renesse theorem and all its variants since then concern elliptic diffusions, in fact this condition plays no role here.}  Take the ordinary differential equation $\dot x_t = -\na U(x_t)$ for a convex $U$: indeed there is contraction for two solutions starting with different initial conditions. More interestingly perhaps, it means that the almost sure contraction   results established  for the Langevin diffusion \eqref{eq:Langevin_cinetique} in fact apply to non-equilibrium chains of interacting particles, in the spirit of those studied in \cite{Menegaki,Eckmann,ReyBellet,HairerHeatBath}, here in the convex case and with a global damping.

One of the main purposes of the present article is to clarify the points mentioned in this section. More precisely, for a diffusion with a constant diffusion matrix, the question is to see what can be said about an almost sure contraction along the parallel coupling, regardless of any reversibility or hypoellipticity assumption.

\subsection{Main contributions and organization}

To conclude this introduction, we sum up the main contributions of this work:
\begin{itemize}
\item Some  characterizations of the contraction property for diffusion processes on $\R^d$ with constant diffusion matrices are provided. This is Theorem~\ref{thm:contraction} below.
\item We establish sharp sufficient and necessary conditions on $U$ and $\gamma$ for a contraction property to hold for the Langevin process \eqref{eq:Langevin_cinetique} in Proposition~\ref{prop:NSconditionLangevin}. This provides quantitative convergence and concentration inequalities for both pinned and unpinned non-equilibrium  anharmonic chains of oscillators, see Propositions~\ref{prop:out-of-eq1} and \ref{prop:out-of-eq2}.
\item We prove a contraction property for the so-called Generalised Langevin processes (see \eqref{eq:EDS} below)  in Theorem~\ref{thm:contractLangevin} under conditions that cover very degenerated cases. It is shown in Proposition~\ref{prop:L2} how to deduce from this a decay in $L^2$ (without  regularization arguments).
\end{itemize}
The remainder of this article is organised in three parts that are mostly independent.  Section~\ref{Sec:courbure_general} is concerned, in a general framework, with the question of almost sure   contraction. It is decomposed as Section~\ref{Subsec:resultat_courbure} where the main result of this part (Theorem~\ref{thm:contraction}) is stated and proven;  Section~\ref{SubSec:concentration} where classical arguments for obtaining concentration inequalities from   Theorem~\ref{thm:contraction} are adapted; Section~\ref{Subsec:PDMP} where a non-diffusive example, for which Theorem~\ref{thm:contraction} is false, is discussed. The second part, Section~\ref{sec:dampedHD}, concerns the standard Langevin process and chains of oscillators:  we will check in Section~\ref{Sec:suffit-pas} that the convexity of $U$ is not sufficient to get a contraction, and present the results for the chains of oscillators in Section~\ref{Sec:chains}.  Third,   Section~\ref{Sec:LangevinGeneralisee} is devoted to the study of generalised Langevin diffusions. After  presentation of these processes (in Section~\ref{SubSec:Generalized}) and some motivations (in Section~\ref{SubSec:LangevinGeneral}),  Section~\ref{Sec:contractLangevinGeneralise} contains the main result of this part (Theorem~\ref{thm:contractLangevin}), which extends the results of \cite{BolleyGuillinMalrieu,Chatterji1,Dalalyan} to some generalised Langevin diffusions, meaning that there is a contraction if $U$ is uniformly convex with bounded Hessian and the friction is high enough; finally, we will see in Section~\ref{Subsec:L2} that the contraction property yields an $L^2$ convergence speed, despite the process being non-reversible.

\subsection*{Notations}

We denote by $\mathcal M_{m,m}^{sym\geqslant 0}(\R)$ (resp. $\mathcal M_{m,m}^{sym> 0}(R)$) the set of real \nvv{symmetric} positive (resp. positive definite) matrices of size $m$, $x\cdot y$ the scalar product between $x,y\in \R^m$ and $|x|=\sqrt{x\cdot x}$. For $M\in\mathcal M_{m,m}^{sym>0}(\R)$ and $x\in \R^m$ we also write $\|x\|_M = \sqrt{x\cdot Mx}$.

For $M,N\in\mathcal M_{m,m,}(\R)$ (not necessarily  symmetric), the notation $M\geqslant N$ (resp. $M>N$) is understood in the sense of quadratic forms, i.e. $x\cdot M x \geqslant x\cdot  N x$ for all  $x\in\R^m$ (resp. $x\cdot M x > x\cdot N x$ for all  $x\neq 0$), or equivalently in sense that the symmetric part $(M+M^T-N-N^T)/2$ of $M-N$ is positive (resp. positive definite) as a symmetric matrix.

We denote by $\mathcal P(\R^m)$ the set of probability measures on  $\R^m$.

\section{Contraction for diffusions}\label{Sec:courbure_general}

This section is devoted to the study of the general case of a  diffusion $(Z_t)_{t\geqslant 0}$ on $\R^d$ that solves
\begin{equation}\label{eq:EDSgenerale}
\dd Z_t \ = \ b(Z_t) \dd t + \Sigma \dd W_t\,,
\end{equation}
where $(W_t)_{t\geqslant 0}$ is a standard Brownian motion on $\R^d$, with:
\begin{assu}\label{assu}
The drift $b\in\mathcal C^\infty(\R^d,\R^d)$ with $\|\partial^{\alpha} b\|_\infty < +\infty$ for all $\alpha\in\N^d$ such that $|\alpha|\geqslant 1$. The diffusion matrix  $\Sigma \in \mathcal M_{d,d}^{sym\geqslant 0}(\R)$  is constant.
\end{assu}

In particular, $b$ is globally Lipschitz. The condition on the higher order derivatives is meant to ensure the validity of some computations with minimal technical considerations. It is not very important since, when it is not satisfied, it may be enough to apply the results with some drifts $(b_n)_{n\geqslant 0}$ that satisfy it and converge to $b$  uniformly on all compact sets and then to let $n\rightarrow +\infty$ (as we will see, our results do not involve these derivatives). 

An important point is that $\Sigma$ is not required to be definite.

The drift being Lipschitz, equation  \eqref{eq:EDS} admits a unique strong non-explosive solution for all initial condition  $z\in\R^d$. We denote by  $(P_t)_{t\geqslant 0}$ the associated Markov semigroup, given by $P_t f(z) = \mathbb E(f(Z_t)|Z_0=z)$ for measurable bounded $f$, $t\geqslant 0$ and $z\in\R^d$. For $k,n\in\N_*$, write $\phi_k(z) = 1/(1+|z|^k)$ and
\[\mathcal C_k^n(\R^d) \ = \ \{f\in\mathcal C^n(\R^d),\ \|\phi_k\partial^\alpha f \|_\infty <+\infty\ \forall \alpha\in\N^d,\ |\alpha|\in \cco 0,n\ccf\}\,.\]
Under Assumption~\ref{assu}, for all $k,n\in\N_*$ there exists $\lambda_{n,k}$ such that for all $f\in\mathcal C_k^n(\R^d)$ and all $t\geqslant 0$, $P_t f \in \mathcal C_k^n(\R^d)$ with$\|P_t f\|_{n,k} \leqslant \exp(\lambda_{n,k}t)\|f\|_{n,k}$, where $\|f\|_{n,k} = \sum_{|\alpha|\leqslant n}\|\phi_k \partial^\alpha f\|_\infty$ (see for instance the proof of \cite[Theorem 2.5]{EthierKurtz}). For all $k\in\N$, $f\in \mathcal C^\infty_k$, $\alpha\in\N^d$ and $t\geqslant 0$, we have $\partial_t (\partial^\alpha P_t f) = \partial^\alpha LP_t f$. 
 Moreover, writing $\psi_\beta(z) = \exp(\beta |z|^2)$, we have that, for all  $t\geqslant 0$, there  exist $\beta_t,M_t>0$ such that for all $s\in[0,t]$ and $z\in\R^d$, $P_s \psi_{\beta_t}(z) \leqslant \exp(M_t(1+|z|^2))$ (this follows for instance from \eqref{eq:GronwallsupBrownien} below).

\subsection{Main result}\label{Subsec:resultat_courbure}

For $z\in\R^d$ we denote by $J_b(z)=(\partial_{z_j}b_i(z))_{1\leqslant i,j\leqslant d}$ the Jacobian matrix of $b$ (the index $i$ being for the line, $j$ for the column). For $M\in\mathcal M_{d,d}^{sym>0}(\R)$ and $\rho\in\R$ we say that the drift $b$ induces a right-contraction of the metric $M$ at rate $\rho$ if the following condition holds:
\begin{equation}
\forall z\in\R^d,\qquad MJ_b(z) \leqslant - \rho M\tag{$\Contd$}
\end{equation}
(recall that   $A\leqslant B$ for possibly non-symmetric matrices is understood in the sense of the associated quadratic forms). The goal of this section is to study the consequences of the condition $\Contd$. Actually, the word \emph{contraction} is disputable when $\rho\leqslant 0$, but we are mainly interested in the case   $\rho>0$. 

Let us start with a few simple remarks. For $M\in\mathcal M_{d,d}^{sym>0}(\R)$ and $\rho\in\R$,  we say that the drift $b$ induces a left-contraction of the metric $M$ at rate $\rho$ if the following condition holds:
\begin{equation}
\forall z\in\R^d,\qquad J_b(z)M \leqslant - \rho M\,.\tag{$\Contg$}
\end{equation}
Then, for given $M,\rho$, the conditions $\Contd$ and $\nvv{\mathrm{Cont}_L}(M^{-1},\rho)$ are equivalent, since
\[x\cdot M J_b(z) x \ = \ (Mx)\cdot J_b(z)M^{-1} Mx\qquad \text{and}\qquad x\cdot  M x \ = \ (Mx)\cdot M^{-1}Mx\]
for all $x,z\in\R^d$. Similarly, writing that, for $z,y\in\R^d$,
\[(z-y)\cdot M \po b(z)-b(y)\pf \ = \ (z-y)\cdot M \po \int_0^1 J_b(sz+(1-s)y) \dd s \pf (z-y)\,,\]
we see that $\Contd$ is of course equivalent to
\begin{equation}\label{eq:contract_integre}
\forall z,y\in\R^d\,,\qquad (z-y)\cdot M \po b(z)-b(y)\pf \ \leqslant \ -\rho \|z-y\|^2_M\,,
\end{equation}
the converse implication being proven by taking  $y=z+\varepsilon x$ with $\varepsilon\rightarrow 0$. Moreover, for $Q\in\mathcal M_{d,d}^{sym>0}(\R)$, notice that  $b$ satisfies  $\Contd$ with $M=Q^T Q$ if and only if $\tilde b$ given by
\[\tilde b(z) \  = \ Q b\po Q^{-1}z\pf\,,\]
satisfies $\nvv{\mathrm{Cont}_R}(I_d,\rho)$. In other words, we can always reduce the study to the case $M=I_d$ through the change of variable $z\leftarrow M^{1/2}z$. That being said, in the following, we will still state the result with a general $M$ as, in many cases, the process of interest is the initial $Z_t$, and several matrices $M$ may be involved simultaneously.  

Finally, under the condition $\Contd$, for any fixed $z\in \R^d$, the deterministic flow that solves $\dot y_t = J_b(z)y_t$ is such that $\|y_t\|_M \leqslant e^{-\rho t} \|y_0\|_M$. This implies in particular that the eigenvalues of $J_b(z)$ necessarily have reals parts lower than $-\rho$, for all $z\in\R^d$ (this condition is not sufficient to get $\Contd$\nvv{, see Proposition~\ref{prop:NSconditionLangevin} and the remarks afterwards}).

\medskip


Recall the definition of the   Wasserstein distances (here associated to Euclidean distances). For $M\in\mathcal M_{d,d}^{sym>0}(\R)$, writing $\mathrm{dist}_M:\R^d\times\R^d  \ni (z,y) \mapsto \|z-y\|_M\in \R_+$, we define for $r\in[ 1,+\infty]$ 
 and $\nu,\mu \in \mathcal P(\R^d)$ 
\[\mathcal W_{M,r}(\nu,\mu) \ = \ \inf_{\pi \in \Pi(\nu,\mu)} \|\mathrm{dist}_M\|_{L^r(\pi)}\]
where $\Pi(\nu,\mu)$ is the set of transference plan between $\nu$ and $\mu$, namely the set of probability laws on $\R^d\times\R^d$ with $d$-dimensional marginals $\nu$ and $\mu$. In other words, the law of a random variable    $(Z,Y)\in\R^d\times\R^d$ belongs to $\Pi(\nu,\mu)$ if $Z\sim \nu$ and $Y\sim \mu$. Such a random variable is called a coupling\footnote{The word coupling is sometimes also used to name $\pi$, instead of transference (or transport) plan.} of $\nu$ and $\mu$. 
Note that we have not restrained the definition to probability measures with finite moments of order $r$,  so that possibly $\mathcal W_{M,r}(\nu,\mu)=+\infty$. There always exists $\pi\in\Pi(\nu,\mu)$ for which the infimum is attained \cite{Villani}.

For a measurable function $f$ on $\R^d$  we write
\[\|\na f(z)\|_{M^{-1}}  \ := \ \lim_{r \downarrow 0}\po \sup \left\{ \frac{|f(z)-f(y)|}{\|y-z\|_M},\ y\in\R^d,\ 0<\|z-y\|_{M}\leqslant r\right\}\pf\,.\]
In particular, if $f$ is differentiable in $z$, this definition is consistent with  $\|\na f(z)\|_{M^{-1}} = \sqrt{\na f(z) \cdot M^{-1}\na f(z)}$. We also write
\[\|\na f\|_{M^{-1},\infty}=\sup \left\{ \frac{|f(z)-f(y)|}{\|y-z\|_M},\ z,y\in\R^d,\ z\neq y\right\} \]
 and $\mathcal C_{b,L}(\R^d)$ the set of bounded Lipschitz functions on $\R^d$.

\begin{thm}\label{thm:contraction}
Under Assumption~\ref{assu}, for $M\in\mathcal M_{d,d}^{sym>0}(\R)$ and $\rho\in\R$, the following propositions are equivalent:
\begin{enumerate}[label=(\roman*)]
\item\label{item_contraction_Contd} $\Contd$.
\item\label{item_contraction_couplage} For all $z_1,z_2\in\R^d$, if $(Z_{i,t})_{t\geqslant 0}$ for  $i=1,2$ are the solutions \eqref{eq:EDSgenerale} driven by a same Brownian motion $W$ with initial conditions $Z_{i,0} = z_i$ then, almost surely,
\[\forall t\geqslant 0,\ \qquad \|Z_{1,t} - Z_{2,t} \|_{M} \ \leqslant \ e^{-\rho t}\|z_1-z_2\|_M\,.\]
\item\label{item_contraction_Winfini} For all $\nu,\mu\in\mathcal{P}(\R^d)$ and $t\geqslant 0$,
\[\mathcal W_{M,\infty}\po \nu P_t,\mu P_t\pf \ \leqslant \ e^{-\rho t}\mathcal W_{M,\infty}(\nu,\mu)\,.\]
\item\label{item_contraction_W2} For all $\nu,\mu\in\mathcal{P}(\R^d)$ and $t\geqslant 0$,
\begin{equation}\label{eq:contractW2}
\mathcal W_{M,1}\po \nu P_t,\mu P_t\pf \ \leqslant \ e^{-\rho t}\mathcal W_{M,1}(\nu,\mu)\,.
\end{equation}
\item\label{item_contraction_gradP_t} For all $f\in \mathcal C_{b,L}(\R^d)$, $t\geqslant 0$ and $z\in\R^d$,
\begin{equation}\label{eq:gradient/semigroupe}
\|\na P_t f(z)\|_{M^{-1}} \ \leqslant \ e^{-\rho t} P_t \po \|\na  f\|_{M^{-1}} \pf(z)\,.
\end{equation}
\item\label{item_contraction_gradP_t2} For all $f\in \mathcal C_{b,L}(\R^d)$ and $t\geqslant 0$,
\[\|\na P_t f\|_{M^{-1},\infty} \ \leqslant \ e^{-\rho t} \|\na  f\|_{M^{-1},\infty}  \,.\]
\end{enumerate}
\end{thm}

\noindent\textbf{Remarks:}
\begin{itemize}
\item Some implications are easy, like $\ref{item_contraction_Winfini} \Rightarrow \ref{item_contraction_W2}$ and $\ref{item_contraction_gradP_t} \Rightarrow \ref{item_contraction_gradP_t2}$.  Except for establishing $\ref{item_contraction_Contd} \Rightarrow \ref{item_contraction_couplage}$, the fact that the process \eqref{eq:EDSgenerale} is a diffusion is only used to close the cycle with the implication $\ref{item_contraction_gradP_t2} \Rightarrow \ref{item_contraction_Contd}$. 
\item The equivalences $ \ref{item_contraction_Winfini}\Leftrightarrow \ref{item_contraction_gradP_t}$ and $\ref{item_contraction_W2} \Leftrightarrow \ref{item_contraction_gradP_t2}  $ are true in a much more general case according to   \cite[Theorem 2.2]{Kuwada2}. In particular, the factor $e^{-\rho t}$ may be replaced by any $c(t)$, for instance $C e^{-\rho t}$ with $C>1$.  However, the fact that $c(0)=1$ for $c(t)=e^{-\rho t}$ is crucial for the implication $\ref{item_contraction_gradP_t2} \Rightarrow \ref{item_contraction_Contd}$. Notice that we won't use the results of Kuwada; more precisely we will recall his arguments for the implication $ \ref{item_contraction_Winfini}\Rightarrow \ref{item_contraction_gradP_t}$  which is the easy half, and we won't prove directly the converse implication which, in our particular case, will ensue from the whole chain of proven implications.  
\item The proof of the implication $\ref{item_contraction_Contd} \Rightarrow \ref{item_contraction_couplage}$ works in very general cases, possibly non Markovian, for instance if $Z_{1,t}$ and $Z_{2,t}$ solves an equation of the form  
\begin{equation*} 
\dd Z_{i,t} \ = \ b(Z_{i,t}) \dd t + \dd X_t
\end{equation*}
where $(X_t)_{t\geqslant 0}$ is any stochastic process (the Markovian case corresponding then to the case where $(X_t)_{t\geqslant 0}$ is a Levy process). See also   Section~\ref{Subsec:PDMP} for another example.
\end{itemize}

\begin{proof}
We will prove the implications $\ref{item_contraction_Contd} \Rightarrow \ref{item_contraction_couplage} \Rightarrow \ref{item_contraction_Winfini} \Rightarrow \ref{item_contraction_gradP_t} \Rightarrow \ref{item_contraction_gradP_t2} \Rightarrow \ref{item_contraction_Contd}$, as well as $ \ref{item_contraction_Winfini} \Rightarrow \ref{item_contraction_W2} \Rightarrow \ref{item_contraction_gradP_t2}$, which will conclude. We will also give a direct proof of  $\ref{item_contraction_couplage} \Rightarrow \ref{item_contraction_Contd}$, which give an autonomous proof of the equivalence   $\ref{item_contraction_Contd} \Leftrightarrow \ref{item_contraction_couplage}$. Finally, we will give a formal proof of   $\ref{item_contraction_Contd}\Rightarrow\ref{item_contraction_gradP_t}$ by a semigroup interpolation argument in the spirit of  Bakry-Emery $\Gamma$ calculus.

For each  implication, without loss of generality (by the change of variable $z\leftarrow M^{1/2}z$) we only consider the case   $M=I_d$. We write $\mathcal W_p = \mathcal W_{I_d,p}$ and $|\na f|=\|\na f\|_{I_d}$.

\medskip

\noindent\textbf{Implication $\ref{item_contraction_Contd} \Rightarrow \ref{item_contraction_couplage}$.}  The parallel coupling $(Z_{1,t},Z_{2,t})_{t\geqslant 0}$ gives, denoting $Z_{\Delta,t} = Z_{1,t}-Z_{2,t}$,
 \begin{eqnarray*}
 \dd \po   | Z_{\Delta,t} |^2  \pf  & = & 2    Z_{\Delta,t} \cdot  \po b(Z_{1,t}) - b(Z_{2,t})\pf \dd t\\
 & = & 2   Z_{\Delta,t} \cdot  \int_0^1 J_b\po s Z^1_t + (1-s)Z^2_t\pf Z_{\Delta,t} \dd s \dd t\\
 & \leqslant & - 2 \rho |Z_{\Delta,t}|^2 \dd t\,,
 \end{eqnarray*}
in other words, almost surely, $\dd (e^{\rho t} |Z_{\Delta,t}|^2) \leqslant 0$ for all $t\geqslant 0$, which concludes.
 
 \medskip

\noindent\textbf{Implication $\ref{item_contraction_couplage} \Rightarrow \ref{item_contraction_Contd}$.}   We fix $z_1,z_2\in\R^d$ and consider the parallel coupling of two  diffusions $(Z_{1,t},Z_{2,t})$  with respective initial conditions   $z_1$ and $z_2$, driven by the same Brownian motion $W$. Write $m_t = \sup_{s\in[0,t]}|\Sigma W_s|$. First, notice that, for all   $t\geqslant 0 $ and $\varepsilon>0$, the event $\mathcal G(t,\varepsilon) = \left\{m_t\leqslant\varepsilon\right\}$ has a non-zero probability. Moreover, since
 \[Z_{i,s} \ = \ z_i +\int_0^s b(Z_{i,u})\dd u + \Sigma W_s\]
for $i=1,2$ and $s\geqslant 0$,  we get  
 \[\forall s\in[0,t],\qquad |Z_{i,s}-z_i| \ \leqslant\  t|b(z_i)| + c\int_0^s |Z_{i,u}-z_i|\dd u + m_s\]
with $c=\|J_b\|_{\infty}$, and thus, almost surely
 \begin{equation}\label{eq:GronwallsupBrownien}
\sup_{s\in[0,t]}|Z_{i,s} - z_i| \leqslant (|b(z_i)| t+m_t) e^{ct}\,. 
 \end{equation}
  Let $t_0\leqslant 1/2$ be small enough so that, under the event $\mathcal G(t,1/2)$, the right hand side is smaller than $1$ for   $i=1,2$.
 Denote by $\mathcal B(x,r)$ the Euclidean ball of $\R^d$ with center $x$  and radius $r$, and $\Phi(x,y)=(x-y) \cdot (b(x)-b(y))$. Let $C>0$ be such that for all  $y_i\in \mathcal B(z_i,1)$, $i=1,2$,
 \[  \left|\Phi(y_1,y_2)- \Phi(z_1,z_2)\right| \ \leqslant \ C\po |y_1-z_1|+|y_2-z_2|\pf\,. \]
 For all $t\in[0,t_0]$, under the event $\mathcal G(t,t) $,
 \[\left|t\Phi(z_1,z_2) - \int_0^t \Phi(Z_{1,s},Z_{2,s})\dd s \right|\ \leqslant \ C t^2 (|b(z_1)|+|b(z_2)| +1) e^{ct} \ \leqslant \ C't^2\]
 for a deterministic $C'>0$. As a consequence, for $t\in[0,t_0]$, the event  $\mathcal G(t,t)$ implies that
 \[\left| |Z_{\Delta,t}|^2- |Z_{\Delta,0}|^2 + 2 t\Phi(z_1,z_2)\right|  \ \leqslant \ 2C' t^2\,.\]
 On the other hand, the event
 \[\mathcal H \ :=\ \left\{\forall t\geqslant 0,\ |Z_{\Delta,t}|^2 \leqslant e^{-2\rho t}|Z_{\Delta,0}|^2\right\}\]
 having by assumption a probability 1, the event  $\mathcal H \cap \mathcal G(t,t)$ has a non-zero probability for all $t\in[0,t_0]$. Yet, this event implies that  
 \[|z_1-z_2|^2 - 2 t\Phi(z_1,z_2) \ \leqslant \ e^{-2\rho t}|z_1-z_2|^2 + C't^2\,, \]
which is a completely deterministic condition. It is thus satisfied for all   $t\in[0,t_0]$, and letting $t$ go to  $0$ we get that $\Phi(z_1,z_2) \leqslant - \rho |z_1-z_2|^2$, which is exactly the integrated version \eqref{eq:contract_integre} of $\Contd$.
 
  \medskip

\noindent\textbf{Implication $\ref{item_contraction_couplage} \Rightarrow \ref{item_contraction_Winfini}$.} It is sufficient to consider the case $\mu =\delta_{z_1}$ and $\nu=\delta_{z_2}$ for all $z_1,z_2\in\R^d$, the general case following by conditioning with respect to the initial condition. Given    $z_1,z_2\in\R^d$ and a Brownian motion $W$, we consider the parallel coupling $(Z_{1,t},Z_{2,t})_{t\geqslant 0}$ driven by  $W$. Let $t\geqslant 0$. Almost surely, $|Z_{1,t}-Z_{2,t}|\leqslant e^{-\rho t}|z_1-z_2|$, and since $(Z_{1,t},Z_{2,t})$ is a coupling of $\delta_{z_1}P_t$ and $\delta_{z_2}P_t$, this implies that $\mathcal W_\infty(\delta_{z_1}P_t,\delta_{z_2}P_t) \leqslant e^{-\rho t}|z_1-z_2|$.

  \medskip

\noindent\textbf{Implication $\ref{item_contraction_Winfini} \Rightarrow \ref{item_contraction_W2}$.} As above, it is sufficient to treat the case of Dirac masses, for which   $\mathcal W_1(\delta_{z_1},\delta_{z_2}) = \mathcal W_{\infty}(\delta_{z_1},\delta_{z_2}) =|z_1-z_2|$. Besides, for all  $\nu,\mu \in \mathcal P(\R)$ and all transference plan $\pi\in\Pi(\nu,\mu)$, $\|\mathrm{dist}\|_{L^1(\pi)} \leqslant \|\mathrm{dist}\|_{L^\infty(\pi)} $, which implies that $\mathcal W_1(\nu,\mu) \leqslant \mathcal W_{\infty}(\nu,\mu)$ for all $\nu,\mu\in\mathcal P(\R^d)$. The conclusion follows from
\[\mathcal W_1\po \delta_{z_1}P_t,\delta_{z_2} P_t\pf  \leqslant \mathcal W_\infty\po \delta_{z_1}P_t,\delta_{z_2} P_t\pf \leqslant e^{-\rho t}|z_1-z_2|\,.\]
 
  \medskip

\noindent\textbf{Implication $\ref{item_contraction_Winfini} \Rightarrow \ref{item_contraction_gradP_t}$.} We recall the proof of   \cite[Proposition 3.1]{Kuwada1}. Fix  $f\in\mathcal C_{b,L}(\R^d)$ and, for $r>0$ and $z\in\R^d$, let
\begin{equation}\label{eq:defGr}
G_r(z) \ = \ \sup_{y\in\mathcal B(z,r)\setminus\{z\}} \frac{|f(y)-f(z)|}{|y-z|}\,,
\end{equation}
so that for all $y,z\in\R^d$,
\[|y-z|\leqslant r \qquad \Rightarrow \qquad |f(y)-f(z)|\leqslant r G_r(z)\,.\]
Fix  $t\geqslant 0$, $y,z\in \R^d$ and $\pi_{y,z}\in\Pi(\delta_y P_t,\delta_z P_t)$ such that $|Y-Z|\leqslant e^{-\rho t}|z-y|$ almost surely with $(Y,Z)\sim \pi_{y,z}$. Then, we have
\begin{eqnarray*}
\left| P_t f(y) - P_t f(z)\right|  & \leqslant & \mathbb E \po |f(Y)-f(Z)|\pf \\ & \leqslant &   e^{-\rho t}|y-z|\mathbb E \po G_{e^{-\rho t}|y-z|}(Z)\pf \ = \ e^{-\rho t}|y-z| P_t\po  G_{e^{-\rho t}|y-z|}\pf (z)\,.
\end{eqnarray*}
Since $f$ is Lipschitz, $G_r(z)$ is uniformly bounded by  $\|\na f\|_\infty$, and by the dominated convergence theorem we get the convergence of  $P_t\po  G_{e^{-\rho t}|y-z|}\pf (z)$ toward $P_t(|\na f|)(z)$ as $y\rightarrow z$, which concludes.

  \medskip

\noindent\textbf{Implication $\ref{item_contraction_gradP_t} \Rightarrow \ref{item_contraction_gradP_t2}$.} Straightforward since $\|P_t  |\na  f|  \|_\infty \leqslant \|   \na  f\|_{\infty} $. 

  \medskip

\noindent\textbf{Implication $\ref{item_contraction_gradP_t2} \Rightarrow \ref{item_contraction_Contd}$.}  We adapt \cite[Proof of $(v)\Rightarrow (i)$]{VonRenesseSturm}. By contradiction, we assume   $\ref{item_contraction_gradP_t2}$ and that there exist $z,x\in\R^d$ with $|x|=1$ such that $x\cdot J_b(z)x = -(\rho-\varepsilon)$ for some $\varepsilon>0$. Let $\mathcal U$ be a neighborhood of $z$ such that $x\cdot J_b(y)x \geqslant -(\rho-\varepsilon/2)$ for all $y\in\mathcal U$. Let $\theta\in\mathcal C^\infty(\R)$ be such that $\theta(r)=r$ for $r\in[-1,1]$, $\|\theta'\|_\infty \leqslant 1$, $\theta(r)=0$ for $r\notin [-3,3]$ (in particular, $\|\theta^{(k)}\|_\infty < +\infty$ for all $k\geqslant 0$). Consider the function $f_x$ on $\R^d$ given by $f_x(y) = \theta((y-z)\cdot x)$ for all $y\in\R^d$, which is $\mathcal{C^\infty}$ and $1$-Lipschitz. From $\ref{item_contraction_gradP_t2} $, for all  non-negative and compactly supported $\varphi \in \mathcal C^\infty(\R^d)$,
\[K_{\varphi}(t) \ :=\ \int_{\R^d} |\na P_t f_x|^2(y) \varphi(y)\dd y \ \leqslant \ e^{-2\rho t} \int_{\R^d}   \varphi(y)\dd y \,.\]
Besides, for all $y\in\mathcal B(z,1/2)$, $\na f_x(y) = x$ and $\na^2 f_x(y)=0$. Fixing $\varphi\in\mathcal C^\infty(\R^d)$ non-zero, non-negative with compact support included in $\mathcal U\cap \mathcal B(z,1/2)$, we get that  $K_\varphi(0) = \int_{\R^d} \varphi(y)\dd y$ and
\begin{align*}
K_\varphi'(0) &= 2 \int_{\R^d} \na   f_x(y) \cdot \na L   f_x(y) \varphi(y)\dd y \\
& = 2 \int_{\R^d}    \na   f_x(y) \cdot  \po J_b(y) \na f_x(y) + L  (\na f_x)(y)\pf \varphi(y)\dd y\\
& \geqslant -2 (\rho-\varepsilon/2) \int_{\R^d} \varphi(y)\dd y \ = \  -2 (\rho-\varepsilon/2) K_\varphi(0)\,. 
\end{align*}
This is in contradiction with the fact  $K_\varphi(t) \leqslant e^{-2\rho t}K_\varphi(0)$ for all  $t\geqslant 0$, which concludes.

  \medskip

\noindent\textbf{Implication $\ref{item_contraction_W2} \Rightarrow \ref{item_contraction_gradP_t2}$.} This is similar to $\ref{item_contraction_Winfini} \Rightarrow \ref{item_contraction_gradP_t}$, and simpler (here this is the classical case of the Kantorovich-Rubinstein duality, see \cite{Villani}). For $f\in\mathcal C_{b,L}(\R^d)$ and $z,y\in\R^d$,   $\pi_{y,z}\in\Pi(\delta_y P_t,\delta_z P_t)$ and $(Y,Z)\sim\pi_{y,z}$,
\[\left| P_t f(y) - P_t f(z)\right|  \leqslant \mathbb E \po |f(Y)-f(Z)| \pf \leqslant \|\na f\|_\infty \mathbb E \po |Y-Z| \pf \,.
\]
The conclusion follows from taking the infimum over $\Pi(\delta_y P_t,\delta_z P_t)$ and using  $\ref{item_contraction_W2} $ and the fact $\mathcal W_1(\delta_y,\delta_z)=|y-z|$.


  \medskip

\noindent\textbf{Implication $\ref{item_contraction_Contd} \Rightarrow \ref{item_contraction_gradP_t}$.} The semigroup interpolation method that directly gives this implication is interesting, and thus we present it here. Nevertheless, since it is not necessary for the completeness of our proof, we will only give a formal proof, without considering some questions of differentiability and integrability   (in particular when $|\na P_s f(z)|$ vanishes).

 We follow the approach of \cite[Lemma 2.4]{Bakry} in the classical case of Bakry-Émery calculus (see also the more recent  \cite[Proposition 3.6]{Menegaki} in a non-reversible case). Fix $f\in\mathcal C_{b,L}(\R^d)$, $t\geqslant 0$ and $z\in\R^d$.  For $s\in [0,t]$, we set
 \[\varphi(s) \ = \ P_s |\na P_{t-s}f|(z) ,\]
 which interpolates between $\varphi(0)=|\na P_t f|$ and $\varphi(t) = P_t|\na f|$. Writing $g=P_{t-s}f$, we compute
\[\varphi'(s) \ = \  P_s\co L \po  |\na g| \pf +   \frac{\partial_s \po |\na g|^2 \pf}{2 |\na g |} \cf \,,\]
with
 \begin{eqnarray}\label{eq:revision}
  \partial_s \po |\na g|^2 \pf \ = \   -2 (\na g)^T \na( Lg)  &=& -2 (\na g)^T  J_b^T \na g  - 2 (\na g)^T L \po    \na g\pf \\
 &\geqslant & 2\rho | \na g|^2  - 2 (\na g)^T L \po    \na g\pf\nonumber \,.
 \end{eqnarray}
On the other hand, the diffusion property of $L$ ensures that, for $a(s)=\sqrt s$,
 \[L \po a(|\na g|^2)\pf \ = \ a'\po |\na g|^2\pf  L\po |\na g|^2\pf  + a''\po |\na g|^2\pf \Gamma\po |\na g|^2\pf \]
with $\Gamma(f)=1/2L(f^2)-fLf$ the carré du champ associated to $L$, here $\Gamma(f) = |\Sigma \na f|^2/2$. 
 At this stage, we have obtained that
 \begin{eqnarray*}
 \varphi'(s) &\geqslant & \rho \varphi(s) +  P_s\co  \frac{L\po |\na g|^2\pf - 2 (\na g)^T L \po    \na g\pf }{2 |\na g|}     - \frac{\Gamma\po |\na g|^2\pf}{4|\na g|^{3}} \cf\,.
 \end{eqnarray*}
 It only remains to prove that, for all $g$, 
 \begin{equation}\label{eq:demoGamma}
2|\na g|^2 \po L\po |\na g|^2\pf - 2 (\na g)^T L \po    \na g\pf \pf\ \geqslant \ \Gamma\po |\na g|^2\pf\,,
 \end{equation}
 since this will imply $\varphi'(s) \geqslant \rho \varphi(s)$ for all  $s\in[0,t]$, hence $\varphi(t) \geqslant e^{\rho t}\varphi(0)$, which will conclude. 
 Yet,
\[
 L\po |\na g|^2\pf - 2 (\na g)^T L \po    \na g\pf \  = \ 2\sum_{i=1}^d \Gamma(\partial_{z_i} g)\  = \ \sum_{i=1}^d \left|\partial_{z_i} \Sigma \na g\right|^2\,,
 \]
 while
 \[\nabla \po |\na g|^2\pf \ = \ 2\sum_{i=1}^d \partial_{z_i} g\ \partial_{z_i} \na g\,,\]
 so that
 \begin{equation}\label{eq:revision2}
 \Gamma\po |\na g|^2\pf \ = \ 2 \left|\sum_{i=1}^d \partial_{z_i} g\ \partial_{z_i} \Sigma\na g\right|^2\,.
 \end{equation}
Notice that here we used the fact that  the diffusion matrix is constant, so that  $\Sigma \na$ commutes with $\nabla$, as in the standard case of the Bakry-Émery calculus where the fact that we work with   $\Gamma$ instead of $|\na \cdot|^2$ yields a similar commutation. The inequality \eqref{eq:demoGamma} then reads
\[\sum_{j=1}^d\left|\sum_{i=1}^d a_i x_{i,j}\right|^2 \ = \ \left|\sum_{i=1}^d a_i x_i\right|^2 \ \leqslant \ \po \sum_{i=1}^d a_i^2\pf  \sum_{i=1}^d |x_i|^2 \ = \  \po \sum_{i=1}^d a_i^2\pf \sum_{j=1}^d\sum_{i=1}^d |x_{i,j}|^2 \]
with $a_i = \partial_{z_i} g(z) \in \R$ and $x_i =(x_{i,j})_{j\in\cco 1,d\ccf} = \partial_{z_i}\Sigma \nabla g(z)\in\R^d$  for  $i\in\cco 1,d\ccf$, hence ensues from  the Cauchy-Schwarz inequality.
\end{proof}

\noindent\textbf{Remarks:}

\begin{itemize}
\item As far as contractions of Wasserstein distances are concerned, we only wrote in the theorem the extremal points   $p=1$ and $p=+\infty$, but obviously the proof of the   implication $\ref{item_contraction_Winfini}\Rightarrow\ref{item_contraction_W2}$ also works to prove that for all $p\in]1,+\infty[$, the proposition
\begin{equation}\label{eq:Wp}
\forall \nu,\mu\in\mathcal{P}(\R^d),\ t\geqslant 0,\qquad \mathcal W_{p,M}\po \nu P_t,\mu P_t\pf \ \leqslant \ e^{-\rho t}\mathcal W_{p,M}(\nu,\mu)
\end{equation}
is implied by $\ref{item_contraction_Winfini}$ and implies $\ref{item_contraction_W2}$. Similarly, for $q\in]1,+\infty[$, the proposition
\begin{equation}\label{eq:gradq}
\forall f\in \mathcal C_{b,L}(\R^d),\ t\geqslant 0,\ z\in\R^d,\qquad \|\na P_t f(z)\|_{M^{-1}}^q \ \leqslant \ e^{-q\rho t} P_t \po \|\na  f\|_{M^{-1}}^q \pf(z)
\end{equation}
is implied by $\ref{item_contraction_gradP_t}$ and implies $\ref{item_contraction_gradP_t2}$. The equivalence between \eqref{eq:Wp} and \eqref{eq:gradq} with $1/p+1/q=1$ is here a particular case of    \cite[Theorem 2.2]{Kuwada2}.
\item By comparison with the classical Bakry-Émery calculus, we replaced the   carré du champ $\Gamma(f) = |\Sigma \na f|^2/2$ by an operator $\Phi(f) = |Q\na f|^2$ with a constant matrix $Q$ (that has nothing to do with $\Sigma$). The idea to consider general operators $\Phi$ traces back at least to \cite{Bakry,Ledoux}, but it does not seem to have been exploited before the recent works   \cite{Baudoin,Baudoin2,Menegaki,MoiGamma} on degenerated diffusions. Notice that an important property  has been used for the direct proof of   $\ref{item_contraction_Contd} \Rightarrow \ref{item_contraction_gradP_t}$ \nvv{(specifically, at \eqref{eq:revision} and \eqref{eq:revision2})}, which is that the operators $\Sigma \na $ and $Q\na$ commute.  Such a commutation condition is also present in  \cite[Theorems 9 and 10]{MoiGamma} or \cite[Lemma 2.7]{Baudoin} \nvv{(where, with the notations of \cite{Baudoin}, this commutation appears as $\Gamma(f,\Gamma^Z(f)) = \Gamma^Z(f,\Gamma(f))$)}. Here it is linked to the fact that we have restrained the study to Euclidean metric on the one hand, and diffusions with constant diffusion matrices on the other hand.  Most of the arguments can be straightforwardly adapted to the general case of a Riemannian manifold, provided such a commutation property holds. 

Notice that the equivalence between an inequality of the form $\Phi(P_t f) \leqslant e^{-\rho t}P_t \Phi(f)$ and $\Gamma_{\Phi} \geqslant \rho \Phi$ where $\Gamma_\Phi$ is the operator that naturally arises when considering the interpolation $P_{t-s}\Gamma(P_s f)$ (see e.g. \cite{MoiGamma}), is classical \cite{Bakry}. Combined with Kuwada's duality results, this means that Theorem~\ref{thm:contraction} could be stated as a general theorem where the condition $\ref{item_contraction_Contd}$ is replaced by a curvature condition $\Gamma_\Phi\geqslant \rho \Phi$ where $\Phi$ is related to the gradient of some distance (not necessarily Euclidean), in the spirit of \cite{Baudoin2} (where only the $\mathcal W_2$ distance is considered, which has the particularity that in this case the Gamma calculus behaves well even for non-diffusion processes, as in \cite{MoiPDMP}). Again, this is classical, and, contrary to such a general framework, one of the main points of  Theorem~\ref{thm:contraction} is that, by working only with Euclidean distances and constant diffusion matrices (which covers all the examples we are interested in here, as in Sections~\ref{sec:dampedHD} and \ref{Sec:LangevinGeneralisee}) we are able to prove the equivalence with   $\ref{item_contraction_Contd}$ which is simpler and more explicit than an abstract curvature condition and, moreover, does not involve the diffusion matrix. 
 
\item It is relatively easy to check the condition $\Contd$. Theorem~\ref{thm:contraction} can thus be seen as a tool to deduce some consequences from it. Nevertheless, it can also be quite easy to check the negation of  $\Contd$ for all $M$ (even only by considering the necessary condition that the real parts of all eigenvalues of $J_b(z)$ have to be negative for all $z\in\R^d$), which is not without interest. For instance, for the Langevin diffusion~\eqref{eq:Langevin_cinetique}, if we suppose that $U$ is uniformly convex outside of a compact set, then the equilibrium measure   $\mu$ with density proportional to  $\exp(-U(x)-|y|^2/2)$ satisfies a logarithmic Sobolev inequality. If moreover $\na^2 U$  is bounded, we can deduce that there exist $\rho>0$ and $C\geqslant 1$ such that
\[\forall \nu \in \mathcal P_1(\R^d),\forall t\geqslant 0\qquad \mathcal W_1\po \nu P_t,\mu\pf  \ \leqslant \ C e^{-\rho t} \mathcal W_1\po \nu  ,\mu\pf\,,\]
or even the same with $\mathcal W_2$, see for instance \cite[Theorem 2]{MonmarcheGuillin2020}. This is weaker than \eqref{eq:contractW2} in two ways: first, $C$ may be greater than $1$, second, one of the initial conditions has to be the equilibrium $\mu$. One could possibly wonder if it is possible to strengthen this long-time convergence into a   contraction \eqref{eq:contractW2} (with  $\rho>0$; otherwise, it is easy), but thanks to Theorem~\ref{thm:contraction} we can prove that this is impossible if $U$  
 is not uniformly convex on the whole space, see Section~\ref{Sec:suffit-pas}.

\item  In the same spirit, Theorem~\ref{thm:contraction},  in particular the equivalence between the contractions of $\mathcal W_p$ for all $p\in[1,+\infty]$, can be viewed in the light of contraction results for  $\mathcal W_1$ distances, such as \cite{Eberle} for elliptic diffusions or \cite{EberleGuillinZimmer} for the kinetic Langevin diffusion. These contractions are not obtained with parallel couplings but with reflection couplings (or a combination of both), which enables to take advantage of the Gaussian noise to bring two processes closer (while the parallel coupling cancels out this noise). They also rely on a concave modification of the usual distance.  For instance, for the overdamped diffusion \eqref{eq:overdamped},  a contraction of the form $\mathcal W_{1,\mathrm{dist}}(\nu P_t,\mu P_t) \leqslant e^{-\rho t} \mathcal W_{1,\mathrm{dist}}(\nu P_t,\mu P_t) $ holds for some $\rho>0$ with $\mathrm{dist}$ a distance which is equivalent to the Euclidean distance if $U$ is convex outside some compact set, see \cite{Eberle}. A similar result holds in the kinetic case \eqref{eq:Langevin_cinetique}, see \cite{EberleGuillinZimmer}. As a consequence, we get indeed that 
\[\forall \nu,\mu \in \mathcal P_1(\R^d),\forall t\geqslant 0,\qquad \mathcal W_1\po \nu P_t,\mu P_t\pf  \ \leqslant \ C e^{-\rho t} \mathcal W_1\po \nu  ,\mu\pf\,,\]
even for $\mu$ that is not the equilibrium. Nevertheless, according to Theorem~\ref{thm:contraction}, if $U$ is not convex, necessarily $C>1$.

\end{itemize}

\subsection{Concentration inequalities}\label{SubSec:concentration}

We give here some consequences of Theorem~\ref{thm:contraction}, which are well-known in the classical case of reversible diffusions (replacing the gradient by the  carré du champ), see \cite{BakryGentilLedoux} and references within. 

\begin{prop}\label{prop:logSob}
Under Assumption~\ref{assu}, for $M\in\mathcal M_{d,d}^{sym>0}(\R)$ and $\rho\in\R$, under the condition $\Contd$, for all measurable non-negative function $f$, $t\geqslant 0$ and $z\in\R^d$,
\[P_t\po f\ln f\pf(z) - P_t f(z) \ln P_t f(z) \ \leqslant \ C_t P_t\po  \| \na  \sqrt{f}\|_{M^{-1}}^2 \pf (z)\]
with $C_t := |\Sigma M^{1/2}|^2 (1-e^{-2\rho t})/\rho$ if $\rho\neq 0$ and $C_t := 2 t|\Sigma M^{1/2}|^2 $ if $\rho=0$.
\end{prop}

\begin{proof}
By density, it is sufficient to prove the result for functions of the form   $f=\varepsilon + h$ with $\varepsilon>0$  and $h\in\mathcal C^\infty(\R^d)$ non-negative with compact support. In particular, in this case, $P_tf \geqslant \varepsilon>0$ for all $t\geqslant 0$.  For such an  $f$, $t\geqslant 0$ and $z\in\mathbb R^d$, set
\[\varphi(s) \ = \ P_s \po P_{t-s}f \ln P_{t-s}f\pf\,,\qquad s\in[0,t]\,.\]
As in the proof of $\ref{item_contraction_Contd} \Rightarrow \ref{item_contraction_gradP_t}$ in Theorem~\ref{thm:contraction}, using the diffusion property of the generator    $L$ and writing $g=P_{t-s} f$,
\begin{align*}
\varphi'(s) \ &= \ P_s\po  L(g\ln g) - (1+\ln g) Lg \pf \ = \ P_s\po \frac{\Gamma(g)}{g}\pf
\end{align*}
with the  carr\'e du champ $\Gamma(f) = |\Sigma \na f|^2/2$, where we used that $a''(s)=1/s$ for $a(s)=s\ln(s)$. As a consequence,
\begin{align*}
\varphi(t)-\varphi(0) \  &= \ \int_0^t P_s\po \frac{\Gamma(g)}{g}\pf \dd s   \\
 &\leqslant \ \frac{|\Sigma M^{1/2}|^2}2\int_0^t P_s\po \frac{\|\na g\|_{M^{-1}}^2)}{g}\pf \dd s \\
&\leqslant \ \frac{|\Sigma M^{1/2}|^2}2\int_0^t e^{-2\rho (t-s)} P_s\po \frac{\po P_{t-s}(\|\na f\|_{M^{-1}})\pf^2}{g}\pf \dd s \,.
\end{align*}
Finally, the Cauchy-Schwarz inequality gives
\[\po P_{t-s}(\|\na f\|_{M^{-1}})\pf^2 \ \leqslant \ P_{t-s} (f) P_{t-s} \po \frac{\|\na f\|_{M^{-1}}^2}{f} \pf \,,\]
and we have indeed obtained
\[\varphi(t)-\varphi(0)  \ \leqslant \ \frac{|\Sigma M^{1/2}|^2}2\int_0^t e^{-2\rho (t-s)} \dd s  P_t\po \frac{ \|\na f\|_{M^{-1}}^2}{f}\pf\,.  \]
 
\end{proof}

We say that $\nu\in\mathcal P(\R^d)$ satisfies a logarithmic Sobolev inequality with  constant $C$ for  $M$ if for all  $f\in\mathcal C_{b,L}(\R^d)$ positive,
\[\int_{\R^d} f\ln f \dd \mu - \int_{\R^d} f \dd \mu \ln \po \int_{\R^d} f \dd \mu\pf  \ \leqslant \ C\int_{\R^d}\|\na \sqrt f\|_{M^{-1}}^2 \dd \mu \,. \]
Thus, Proposition~\ref{prop:logSob}  states a logarithmic Sobolev inequality for $\delta_z P_t$ with constant $C_t$, for all $z\in\R^d$ (which is called a \emph{local} logarithmic Sobolev inequality \cite{BakryGentilLedoux}).

\medskip

\noindent\textbf{Remark:} in fact, in the classical case, such an inequality is \emph{equivalent} to a curvature-dimension condition, and is also equivalent to a so-called \emph{inverse} local inequality   \cite[Theorem 5.5.2]{BakryGentilLedoux}. However, to get the equivalence between these notions, it is not sufficient that the carré du champ is bounded above by the square of the gradient, they need to be equal (or at least equivalent), which has no reason to hold in our case. In particular, an inverse logarithmic Sobolev inequality measures a regularisation property of the semigroup, yet our framework covers non-hypoelliptic diffusion (for instance $\dot x_t = - x_t$) for which such an inequality is clearly false.

\medskip

The local inequality stated in   Proposition~\ref{prop:logSob} together with the  gradient/semigroup sub-commutation \eqref{eq:gradient/semigroupe} imply that $\nu P_t$ satisfies a logarithmic Sobolev inequality $C_t+e^{-2\rho t} C$ as soon as $\nu$ satisfies such an inequality with constant $C$ (see e.g. \cite[Theorem 23]{MoiLangevinCinetique}). This has some nice consequences when $\rho>0$, in particular to get non-asymptotic confidence intervals  for ergodic means of the process:

\begin{prop}\label{prop:concentration}
Under Assumption~\ref{assu}, if $\Contd$ is satisfied for some $M\in\mathcal M_{d,d}^{sym>0}(\R)$ with $\rho>0$, then:
\begin{enumerate}
\item $(P_t)_{t\geqslant 0}$ has a  unique invariant probability measure   $\mu_\infty$, which has finite moments of all orders and satisfies a logarithmic Sobolev inequality for $M$ with constant   $ |\Sigma M^{1/2}|^2  /\rho$.
\item If $\nu \in\mathcal P(\R^d)$ satisfies a logarithmic Sobolev inequality for $M$ with  constant $C'$ and if $(Z_t)_{t\geqslant 0}$ is a solution of \eqref{eq:EDSgenerale} with initial distribution  $\nu$, then for all $t_0>0$, $T>0$, $u\geqslant 0$, $n\in\N^*$ and all measurable function $f$ on $\R^d$ such that $\|\na f\|_{M^{-1},\infty} \leqslant 1$,
\begin{align*}
\mathbb{P}\po \frac1n\sum_{k=1}^n  \po f(Z_{kt_0}) - \mathbb E \po f(Z_{kt_0})\pf\pf \geqslant u\pf\  & \leqslant \   \exp \po - \frac{n u^2(1-e^{-\rho t_0})^2}{C_{t_0}+e^{-\rho t_0}C'/n}\pf\\
\left| \frac1n\sum_{k=1}^n    \mathbb E \po f(Z_{kt_0})\pf  - \mu_\infty(f)\right| \ &\leqslant \ \frac{e^{-\rho t_0}}{n(1-e^{-\rho t_0})}\mathcal W_{M,1}\po \nu,\mu_\infty\pf\\
\mathbb{P}\po \frac1T\int_0^T  \po f(Z_{t}) - \mathbb E \po f(Z_{t})\pf\pf \dd t \geqslant u\pf\  & \leqslant \   \exp \po - \frac{T \rho^2 u^2}{2|\Sigma M^{1/2}|^2+C'/T}\pf\\
\left| \frac1T\int_0^T    \mathbb E \po f(Z_{t})\pf \dd t  - \mu_\infty(f)\right| \ &\leqslant \ \frac{1}{\rho T}\mathcal W_{M,1}\po \nu,\mu_\infty\pf\,.
\end{align*}

\end{enumerate}
 
\end{prop}
 
\begin{proof}
For $r\geqslant 1$, denote by $\mathcal P_r(\R^d)$  the set of probability measures on  $\R^d$ with finite moment of order $r$.  Fix $t>0$ and $r\geqslant 2$. Theorem~\ref{thm:contraction} implies  that $P_t$ is a contraction of $(\mathcal P_r(\R^d),\mathcal W_{M,r})$, which is a complete metric space by  \cite[Theorem~6.18]{Villani}, and thus there exists a unique $\mu_\infty \in\mathcal P_r(\R^d)$ invariant for $P_t$. For $s\geqslant 0$, $(\mu_\infty P_s)P_t = (\mu_\infty P_t) P_s = \mu_\infty P_s$, so that $\mu_\infty P_s$ is invariant for $P_t$. By uniqueness $\mu_\infty P_s = \mu_\infty$, in other words $\mu_\infty $ is  invariant for $P_s$ for all $s\geqslant 0$. Applying \eqref{eq:contractW2} with $\nu = \delta_z$   and $\mu = \mu_\infty$ yields the convergence of $\delta_z P_t$ toward $\mu_\infty$ as $t\rightarrow +\infty$ in the sense of the $\mathcal W_1$ distance, which implies the weak convergence. The logarithmic Sobolev inequality for  $\mu_\infty$ is then obtained by letting   $t\rightarrow +\infty$ in Proposition~\ref{prop:logSob}.

We now prove the second point of the   proposition. The discrete-time case of $(Z_{kt_0})_{k\in\N}$ is obtained by applying \cite[Theorem 23]{MoiLangevinCinetique} to the operator $P_{t_0}$ thanks to Theorem~\ref{thm:contraction} and Proposition~\ref{prop:logSob}, after the change of variable $z\leftarrow M^{1/2} z$.  The continuous-time case is then obtained by applying the previous result with  $t_0=T/n$ and letting $n\rightarrow +\infty$. Indeed, on any interval $[0,T]$, $t\mapsto f(Z_t)$ is almost surely uniformly continuous. In other words, for all $\varepsilon>0$, there exists a random variable  $\delta_{\varepsilon}$ which is almost surely positive and such that, almost surely,  
\[\forall t,s\in[0,T],\ |s-t|\leqslant \delta_\varepsilon\ \Rightarrow \ |f(Z_t)-f(Z_s)|\leqslant \varepsilon\,.\]
Since $\mathbb P\po \delta_\varepsilon \leqslant 1/m\pf \rightarrow 0$ as $m\rightarrow +\infty$, we deduce the convergence in probability of   $1/n\sum_{k=1}f(Z_{kT/n})$ toward $1/T\int_0^T f(Z_s)\dd s$.

\end{proof}

We finish this section by an example of application of Proposition~\ref{prop:concentration}. In \cite{Menegaki}, Menegaki proves a quantitative Wasserstein contraction for a  close to harmonic chain of interacting oscillators (different to those considered in Section~\ref{Sec:chains} since there is only friction at the ends of the chain). She already noticed that this result implies a log-Sobolev inequality for the invariant measure. Thanks to  Proposition~\ref{prop:concentration}, we also get concentration inequalities for the model studied by Menegaki.

%
%
%
%

 \subsection{Discussion on a jump process}\label{Subsec:PDMP}

Let us show that Theorem~\ref{thm:contraction} does not hold in general for non-diffusive processes, and let us exemplify the fact that an almost sure contraction property is a very restrictive condition.

Consider the Markov process on  $\R^d$  with generator $\mathcal L$ given by
\[\mathcal L\varphi(z) \ = \ -a z\cdot \na \varphi(z) + \lambda \po \varphi(hz) - \varphi(z)\pf\,,\]
 with constants $a,h\in\R$ and $\lambda>0$. For all  $M\in \mathcal M_{d,d}^{sym>0}(\R)$, the deterministic drift $b(x)=-ax$ satisfies $\mathrm{Cont}_d(M,a)$ and does not satisfy $\mathrm{Cont}_d(M,\rho)$ for any $\rho>a$.
 
The process corresponding to this generator solves 
 \begin{equation}\label{eq:PDMP}
\dd Z_t \ = \ - a Z_t + (h-1)Z_t \dd N_{\lambda t} 
 \end{equation} 
 where $(N_t)_{t\geqslant 0}$ is a standard Poisson process. In other words, the process follows the deterministic flow $\dot z= -az$ and, at rate $\lambda$, is multiplied by $h$. Besides, notice that the process $(QZ_t)_{t\geqslant 0}$ is also a Markov process with generator $\mathcal L$ for all matrices $Q$, so that in the following we will only consider the case $M=I_d$, keeping in mind that all the results stated with the standard Euclidean norm in fact also hold with all norms  $\|\cdot\|_M$ for $M\in\mathcal M_{d,d}^{sym>0}(\R)$.
 
Let $z_1,z_2\in\R^d$ with $z_1\neq z_2$, $(N_t)_{t\geqslant 0}$  a standard Poisson process and $(Z_{1,t},Z_{2,t})_{t\geqslant 0}$ the solutions  of \eqref{eq:PDMP} (driven by $N$) with initial conditions $z_1$ and $z_2$ (namely, $Z_{1}$ and $Z_{2}$ are coupled in order to have the same jump times). This is the analog of the parallel coupling for diffusions (sometimes called synchronous coupling). Notice that, if $|h|>1$, then there exists no  $\rho\in\R$ such that, for a given $t>0$, almost surely, $|Z_{1,t}-Z_{2,t}|\leqslant e^{-\rho t}|z_1-z_2|$, since the number of jumps occurring during the  interval $[0,t]$ can be arbitrarily high with positive probability. Conversely, if $|h|\leqslant 1$, then a jump never increases the distance. In that case, for  $\rho = a$:
 \begin{equation}\label{eq:PDMPcontract}
 \text{almost surely,}\qquad \forall t\geqslant 0\,, \qquad |Z_{1,t}-Z_{2,t}| \leqslant e^{-\rho t}|z_1-z_2|\,.
 \end{equation}
 On the other hand this property is false for all   $\rho>a$ (even when replacing ``$\forall t\geqslant0$" by a fixed value of $t$, or $|\cdot|$ by any other  $\|\cdot\|_M$), since for all interval $[0,t]$ the probability that no jump has occured yet is non-zero, in which case the previous inequality is an equality  (with $\rho=a$).
 
From this, we get that, if $|h|\leqslant 1$, whatever $M$, the points $\ref{item_contraction_Winfini}, \ref{item_contraction_W2}, \ref{item_contraction_gradP_t2}$ and  $\ref{item_contraction_gradP_t2}$ of Theorem~\ref{thm:contraction} are all satisfied with $\rho = a$ (the proof of these facts from \eqref{eq:PDMPcontract} does not rely on the diffusion property of the process).

Let us prove that, for all  $h\in\R$,   the point $\ref{item_contraction_Winfini}$ can never hold with a $\rho>a$. Take $\mu=\delta_{z_1}$ and $\nu=\delta_{z_2}$for some $z_1\neq z_2$. Fix also $t\leqslant \ln(4/3)/\lambda$ and consider  $(Z_{1,t},Z_{2,t})$ any coupling of $\mu P_t$ and $\nu P_t$. The probability that no jump has occurred before time $t$ being larger than $3/4$ for each process, the probability of $\{Z_{i,t}=e^{-at}z_i,\ i=1,2\}$, hence of $\{|Z_{1,t}-Z_{2,t}|=e^{-at}|z_1-z_2|\}$, is non-zero. As a  consequence, for all $t\leqslant \ln(3/4)/\lambda$ and $z_1,z_2\in\R^d$, $\mathcal W_\infty(\delta_{z_1}P_t,\delta_{z_2} P_t) \geqslant e^{-at}|z_1-z_2|$, which concludes.
 
 \medskip
 
 We see that the contribution of $h$ to the contraction is absolutely not taken into account by the almost sure contraction property. Yet, when  $|h|<1$, each jump reduces the distances, which means the contraction is on average larger than $a$. Besides, in the case $|h|>1$, there is also a contraction on average provided $a$ is large enough. This balance between the contraction (or dilatation) properties of a deterministic drift and a jump mechanism have been studied in  \cite{MoiPDMP}  in a more general framework. In the present case, the situation is very simple since, almost surely, the synchronous coupling yields  
 \[Z_{1,t} - Z_{2,t} \ = \ e^{-at} h^{N_{\lambda t}}(z_1-z_2)\,.\]
In particular, for all $n\geqslant 1$,
 \[\mathbb E \po |Z_{1,t} - Z_{2,t} |^n \pf \ = \  e^{-\po an + (1-|h|^n)\lambda\pf t}|z_1-z_2|^n  \,.\]
 Hence, we obtain that, for all $n\geqslant 1$, $t\geqslant 0$ and $\nu,\mu\in\mathcal P(\R^d)$,
 \[\mathcal W_n\po \nu  P_t,\mu P_t\pf \ \leqslant \ e^{-\rho t}\mathcal W_n\po \nu   ,\mu \pf \]
holds with $\rho = \rho_n :=  a  +  (1-|h|^n)\lambda/n  $. Note that, contrary to the case $n=+\infty$, here we haven't proved that the inequality is false for $\rho>\rho_n$.  It is already interesting to notice that, in the expression of $\rho_n$, the respective weight of the deterministic drift and of the jumps depends on $n$. In particular, for $n=1$, we get that the point $\ref{item_contraction_W2}$ of Theorem~\ref{thm:contraction} (and thus the point $\ref{item_contraction_gradP_t2}$ as the implication does not rely on the diffusion property of the process) is satisfied with $\rho =  a  +  (1-|h|)\lambda$, whatever the value of $h\in\R$. For $|h|<1$, we see that the equivalence between points $\ref{item_contraction_Winfini}$ and $\ref{item_contraction_W2}$ is false.

\section{Variations on damped Hamiltonian dynamics}\label{sec:dampedHD}

We focus in this section on diffusions \eqref{eq:EDSgenerale} on $\R^d$ with $d=2n$, $n \in \N$ with a drift of the form (decomposing $z=(x,y)\in \R^d=\R^n\times\R^n$)
 \begin{equation}\label{eq:bsuffitpas}
b(x,y) \ = \ \begin{pmatrix}
 y \\ -\na U(x) - \gamma y
 \end{pmatrix}\,, 
 \end{equation}
 for some $U\in\mathcal C^\infty(\R^n)$, $\gamma>0$. This is for instance the case of the Langevin process \eqref{eq:Langevin_cinetique}, corresponding to the diffusion matrix
 \[\Sigma = \sqrt{2\gamma}\begin{pmatrix}
 0 & 0 \\ 0 & I_n
 \end{pmatrix}\,.\]
 Another example is given by chains of oscillators. Take $n=pN$ where $p\in \N$ is the ambient space dimension and $N\in\N$ is the number of oscillators, so that we decompose $(x,y)=(x_1,\dots,x_N,y_1,\dots,y_N)$ with $(x_i,y_i)\in \R^p \times\R^p$ the position and velocity of the $i^{th}$ particle. Let $V \in \mathcal C^\infty(\R^p)$ be a so-called pinning potential (possibly $V=0$), an even $F\in \mathcal C^\infty (\R^p)$ be a so-called interaction potential, and $T_1,\dots,T_N\geqslant 0$ be some temperatures.  The corresponding chain of oscillators is then the diffusion process $(X_{1,t},\dots,X_{N,t},Y_{1,t},\dots,Y_{N,t})_{t\geqslant 0}$ solution of
 \begin{equation}\label{eq:chains}
\forall i\in\cco 1,N\ccf\,,\qquad \left\{\begin{array}{rcl}
\dd X_{i,t} & = &  Y_{i,t} \dd t \\
\dd Y_{i,t} & = & -   \na_{x_i} U(X_{t})  \dd t     - \gamma Y_{i,t} \dd t + \sqrt{2\gamma T_i} \dd W_{i,t}
\end{array}
\right.
\end{equation}
where $W=(W_1,\dots,W_N)$ is a standard Brownian motion on $\R^{pN}$ and
\begin{equation}\label{eq:pot_chains}
U(x) = \sum_{i=1}^N V(x_i) +   \sum_{i=1}^{N-1} F(x_i-x_{i+1})\,.
\end{equation}
Typically, $T_i=0$ for all $i\in\cco 2,N-1\ccf$. When the $T_i$'s are not all equal, the invariant measure of the process is not explicit.  Remark that, in \eqref{eq:chains} and in contrast to other studies like \cite{Menegaki,Eckmann,ReyBellet,HairerHeatBath} and references within, the friction term $ - \gamma Y_{i,t} \dd t$ concern all the particles, not only those at the ends of the chain ($i=1$ and $N$). The system is said to be pinned if $V\neq 0$ and unpinned otherwise.

More generally, we can consider a total energy of the form
\begin{equation}\label{eq:pot_graph}
U(x) = \sum_{i=1}^N V_i(x_i) + \sum_{i\sim j} F_{i,j}(x_i-x_j)
\end{equation}
with an equivalence relation $i\sim j$ given by a non-oriented graph on $\cco 1,N\ccf$ and some potentials $V_i,F_{i,j}$, as in \cite[Section 4]{MoiGamma}. For clarity we won't discuss this  straightforward generalization in detail.

 \subsection{Convexity of $U$ is not sufficient}\label{Sec:suffit-pas}
 
As discussed in the introduction, a natural question is whether the convexity of $U$ is sufficient to have a contraction for the drift  \eqref{eq:bsuffitpas}.

  
  Consider first the Gaussian case, namely $U(x) = x\cdot Sx$ with $S\in\mathcal M_{n,n}^{sym>0}(\R)$. Up to an orthonormal change of coordinates, we may assume that $S$ is a diagonal matrix, in which case the $n$ couples of coordinates $(X_i,Y_i)_{i\in\cco 1,n\ccf}$ are independent diffusions on $\R\times\R$ (in the Langevin case \eqref{eq:Langevin_cinetique}). Hence, we only need to consider the case $n=1$. In that case,
  \[ J_b(x,y) \ = \ \begin{pmatrix}
0 & 1 \\ -\lambda  & -\gamma
\end{pmatrix}\]
for some given $\lambda>0$, for all $x,y\in\R^n$. The eigenvalues of this matrix are $-\gamma/2\pm\sqrt{\gamma^2/4-\lambda}$ (possibly complex), in any cases their real part is negative. According to \cite[Lemma 2.11]{Arnold}, for all $\varepsilon>0$ there exists $M\in \mathcal M_{2,2}^{sym>0}(\R)$ such that 
\[MJ_b(x,y) \  \leqslant \ -\po \mathfrak{Re}\po \frac{\gamma}2 - \sqrt{\frac{\gamma^2}{4}-\lambda} \pf - \varepsilon\pf M\,, \]
the condition $\varepsilon>0$ being necessary only in the defective case, here when $\gamma^2= 4\lambda$ (in that case there is no contraction at rate $\gamma/2$, a polynomial factor has to be added, see \cite{MoiGamma}). The condition $\Contd$ is thus indeed satisfied with some $\rho>0$ as soon as $\lambda>0$,  and by combining the corresponding matrices $M$ we see that such a condition is always satisfied in dimension $n$ regardless of the eigenvalues of $S$   (and of the relations between them).

In the general (possibly non-Gaussian) case, however, the situation is slightly more complicated, in particular it is not enough to study the eigenvalues of $J_b(x,y)$ for $(x,y)\in\R^{2n}$. The main goal of this section is to prove the following.

\begin{prop}\label{prop:NSconditionLangevin}
Let $U\in \mathcal C^\infty(\R^n)$ and $\gamma>0$.
\begin{enumerate}
\item \nvv{If $b$ given by \eqref{eq:bsuffitpas} satisfies  $\Contd$ for some $\rho>0$ and $M\in\mathcal M_{2n,2n}^{sym>0}(\R)$ then necessarily there exists $\lambda >0$ such that $\na^2 U(x) \geqslant \lambda I_n$ for all $x\in\R^n$. }
\item If  there exist $x,x'\in \R^n$ and $\lambda,\Lambda>0$ such that \nvv{$\sqrt{\Lambda}-\sqrt{\lambda}   \geqslant   \gamma $} and
\begin{equation}\label{eq:Usuffitpas}
\na^2U(x)\ =\ \lambda I_n\,,\qquad   \na^2 U(x') \ = \  \Lambda I_n\,,
\end{equation}
then, for all $\rho>0$ and $M\in\mathcal M_{2n,2n}^{sym>0}(\R)$, $b$ given by \eqref{eq:bsuffitpas} does not satisfy $\Contd$. 
\item If there exist $\lambda,\Lambda>0$ such that \nvv{$\sqrt{\Lambda}-\sqrt{\lambda}  <  \gamma$} and
\begin{equation}\label{eq:condUnabla2}
\forall x\in\R^n\,,\qquad \lambda I_n \leqslant \na^2 U(x) \leqslant \Lambda I_n\,,
\end{equation}
then $b$ given by \eqref{eq:bsuffitpas} satisfies $\Contd$ for some $\rho>0$ and $M\in\mathcal M_{2n,2n}^{sym>0}(\R)$.
\item More precisely, if \eqref{eq:condUnabla2} holds for some $\lambda,\Lambda>0$ with  $\Lambda\leqslant \gamma^2/4$, then $b$ given by \eqref{eq:bsuffitpas} satisfies $\Contd$ with $\rho = \lambda/(3\gamma)$ and
\[M = \begin{pmatrix}
I_n & \gamma^{-1} I_n \\
\gamma^{-1} I_n  & \Lambda^{-1} I_n 
\end{pmatrix}\,,\]
which satisfies $\min(1,\Lambda^{-1})/2 I_{2n} \leqslant M \leqslant 3/2 \max(1,\Lambda^{-1}) I_{2n}$.
\end{enumerate}
\end{prop}

\noindent\textbf{Remarks:} 

\begin{itemize}
\item \nvv{In dimension $n=1$, the second and third points are  \cite[Proposition 3.6]{Arnold}. The condition (3.5) of \cite{Arnold} corresponds to $\Contg$ for the drift \eqref{eq:bsuffitpas}.}
\item To sum up, the condition \nvv{$\sqrt{\Lambda}-\sqrt{\lambda}   \geqslant   \gamma $} is sharp, in the sense that, if it is satisfied, then all potentials $U$ satisfying \eqref{eq:condUnabla2} gives a contraction for the corresponding process while, if it is not satisfied, there exists potentials $U$ that satisfy \eqref{eq:condUnabla2} and for which there is no contraction.
\item In other words, for fixed values of $\lambda$ and $\Lambda$, a contraction (for a Euclidean distance) cannot hold if the friction $\gamma$ is too small, which is consistent with the results of \cite{BolleyGuillinMalrieu,Chatterji1,Dalalyan}. For a given $\Lambda$ the constraint on $\gamma$ gets weaker as $\lambda$ get close to $\Lambda$, which is consistent with the Gaussian case where $\lambda = \Lambda$ and there is no constraint on  $\gamma$. For a fixed $\lambda$, as $\gamma\rightarrow +\infty$,  a contraction holds for larger and larger values of $\Lambda$, which is consistent with the overdamped limit of the process: indeed for a fixed $U$ it is well-known that  $(X_{\gamma t})_{t\geqslant 0}$ converges as $\gamma\rightarrow +\infty$ to the overdamped process \eqref{eq:overdamped}, for which a contraction holds as soon as $U$ is strongly convex.
\item 
By an orthonormal transformation, a similar result holds if the condition~\eqref{eq:condUnabla2} is replaced by $\lambda  S \leqslant \na^2 U(x) \leqslant \Lambda S$ for some $S\in\mathcal M_{n,n}^{sym>0}(\R)$, and similarly if \eqref{eq:Usuffitpas} is replaced by $\na^2 U(x) = \lambda S$, $\na^2 U(x')=\Lambda S$.
\item 
This  also proves that the assumption that $\na^2 U$ is bounded is necessary to get a contraction in the kinetic case. Indeed, for   $n=1$, for a given value of $\gamma$, if we suppose that $U$ is uniformly convex with $U''(x)$  that is not bounded then, by fixing some $x_0$ with $\lambda:=U''(x_0)>0$, the condition \nvv{$\sqrt{U''(x)} - \sqrt{\lambda} \geqslant \gamma $} will be satisfied for some $x$ and we will be in the frame of the second point of  Proposition~\ref{prop:NSconditionLangevin}.
\item \nvv{It also provides a counter-example which shows that it is not sufficient that the eigenvalues of $J_b(z)$ have a real part lower than $-\rho$ uniformly in $z$ for $\Contd$ to hold. Indeed, in dimension $n=1$ with $U(x)=x^2+x^4$ and $\gamma=1$ for instance, the eigenvalues of $J_b(x,y)$ are $-1/2\pm \sqrt{1/4 - U''(x)}$ whose real part is $-1/2$ for all $z=(x,y)$  although, from the previous remark, for any $\rho>0$, $\Contd$ does not hold for any $M$.}
\item We only give an explicit expression for $\rho$ and $M$ under the strong condition $4\Lambda \leqslant \gamma^2$ for simplicity and because, as this condition does not involve $\lambda$, it is particularly interesting in the case of unpinned chains of oscillators, see next section. Of course, as soon as \nvv{$\sqrt{\Lambda}-\sqrt{\lambda}   <  \gamma$}, the proof of Proposition~\ref{prop:NSconditionLangevin} yields  explicit values for $\rho$ and $M$ such that $\Contd$ holds. The question of maximizing $\rho$ for given values of $\Lambda,\lambda$ with a suitable choice of $\gamma$ and $M$ \nvv{leads to} tedious expressions and we won't discuss it here.
\end{itemize}

The remainder of this section is devoted to the proof of Proposition~\ref{prop:NSconditionLangevin}. We start with an elementary lemma.

\begin{lem}\label{lem:condCNS}
For $\lambda,\Lambda,\gamma>0$ with $\Lambda\geqslant \lambda$, the two following are equivalent.
\begin{enumerate}[label=(\roman*)]
\item \nvv{$\sqrt{\Lambda}-\sqrt{\lambda}   <  \gamma$}.
\item There exist $a,c\in\R$ with $c^2<a$ such that, for $\xi\in \{\lambda,\Lambda\}$,
\begin{equation}\label{eq:xi}
\begin{pmatrix}
1 & c \\ c & a
\end{pmatrix}
 \begin{pmatrix}
0 & -1 \\ \xi  & \gamma
\end{pmatrix} \ > \ 0\,.
\end{equation}
\end{enumerate}
Moreover, if $\gamma\geqslant 2 \sqrt{\Lambda}$, setting $a=1/\Lambda$ and $c=1/\gamma$, then $c^2 \leqslant a/4$ and for all $\xi\in[\lambda,\Lambda]$,
\begin{equation*}
\begin{pmatrix}
1 & c \\ c & a
\end{pmatrix}
 \begin{pmatrix}
0 & -1 \\ \xi  & \gamma
\end{pmatrix} \ \geqslant  \ \frac{\lambda}{3\gamma} \begin{pmatrix}
1 & c \\ c & a
\end{pmatrix}\,.
\end{equation*}
\end{lem}
\begin{proof}
Considering the diagonal coefficients and the determinant of the symmetrised matrix, $(ii)$ is equivalent to
\[\exists a,c\in\R\quad \text{s.t.}\quad c^2<a\,,\quad c>0\,,\quad   \gamma a  >c \,, \quad 4c\xi(\gamma a-c) > (a\xi +c\gamma -1)^2\,,\text{ for }\, \xi\in\{\lambda,\Lambda\}\,. \]
By the change of variable $\alpha=\gamma a/c-1$ and $\beta=1/c$, this is equivalent to
\begin{align*}
& \exists \alpha,\beta>0\quad \text{s.t.}\quad (\alpha+1)\beta>\gamma\,, \quad 4\xi\alpha > \po (\alpha+1)\xi/\gamma +\gamma -\beta\pf^2\,,\text{ for }\, \xi\in\{\lambda,\Lambda\}\\
\Leftrightarrow\quad & \exists \alpha>0\quad \text{s.t.}\quad \left(\frac{\gamma}{\alpha+1},+\infty\right)\cap I(\lambda,\alpha)\cap I(\Lambda,\alpha) \neq \emptyset
\end{align*}
 where
 \[I(\xi,\alpha) \ = \ \po(\alpha+1)\xi/\gamma +\gamma - 2\sqrt{\xi \alpha},(\alpha+1)\xi/\gamma +\gamma +2\sqrt{\xi \alpha}\pf\,.\]
 The upper bound of $I(\Lambda,\alpha)$ being always larger than the one of $I(\lambda,\alpha)$,
for any $\alpha>0$,
\begin{align*}
I(\lambda,\alpha)\cap I(\Lambda,\alpha) \neq \emptyset & \quad\Leftrightarrow\quad   (\alpha+1)\Lambda/\gamma +\gamma - 2\sqrt{\Lambda \alpha} < (\alpha+1)\lambda/\gamma +\gamma + 2\sqrt{\lambda \alpha}\\
& \quad\Leftrightarrow\quad  (\alpha+1)(\Lambda-\lambda)/\gamma   < 2\po \sqrt{\Lambda}+\sqrt{\lambda}\pf\sqrt{\alpha} \\
& \quad\Rightarrow\quad   (\Lambda-\lambda)/\gamma   <   \sqrt{\Lambda}+\sqrt{\lambda} \\
& \quad\Leftrightarrow\quad I(\lambda,1)\cap I(\Lambda,1) \neq \emptyset 
\end{align*}
where we used that $\alpha \mapsto \sqrt{\alpha}/(1+\alpha)$ is maximal for $\alpha=1$. Moreover, if $\beta\in I(\xi,1)$ for some $\xi>0$, then
\[\beta > 2\xi/\gamma +\gamma - 2\sqrt{\xi } \geqslant \frac{\gamma}{2}\,,\]
which means that
\[I(\lambda,1)\cap I(\Lambda,1) \neq \emptyset \quad \Rightarrow \quad \left(\frac{\gamma}{2},+\infty\right)\cap I(\lambda,1)\cap I(\Lambda,1) \neq \emptyset\]
and thus
\begin{align*}
\exists \alpha>0\ \text{s.t.}\ \left(\frac{\gamma}{\alpha+1},+\infty\right)\cap I(\lambda,\alpha)\cap I(\Lambda,\alpha) \neq \emptyset \qquad  &  \Leftrightarrow \qquad I(\lambda,1)\cap I(\Lambda,1) \neq \emptyset\\
&  \Leftrightarrow \qquad \Lambda-\lambda   <  \gamma\po \sqrt{\Lambda}+\sqrt{\lambda}\pf\,.
\end{align*}
This concludes the first part of the lemma. For the second part, assume $\gamma^2 \geqslant 4 \Lambda$, set $a=1/\Lambda$ and $c=1/\gamma$. Then, for $\xi\geqslant 0$,
\[
\begin{pmatrix}
1 & c \\ c & a
\end{pmatrix}
 \begin{pmatrix}
0 & -1 \\ \xi  & \gamma
\end{pmatrix}    =  \  \begin{pmatrix}
\frac\xi\gamma  & 0 \\ a\xi  & -\frac1\gamma+a\gamma
\end{pmatrix} \
 \geqslant  \  \begin{pmatrix}
\frac{\xi}{2\gamma}  & 0 \\ 0 & -\frac1\gamma +a\gamma - \frac{a^2\gamma\xi}{2}
\end{pmatrix}\,.
\]
Using that $a\gamma^2 \geqslant 4$, we get for $\xi \in [\lambda,\Lambda]$ that
\[-\frac1\gamma +a\gamma \po 1  - \frac{a\xi}{2} \pf \   \geqslant \  -\frac1\gamma +  \frac{a\gamma}{2}  \ \geqslant\  \frac{a\gamma}{4}\,.\] 
Moreover, since $c^2\leqslant a/4$,
\begin{equation}\label{eq:eqnorm}
\frac12 \begin{pmatrix}
1 & 0 \\ 0 & a
\end{pmatrix} \leqslant \begin{pmatrix}
1 & c \\ c & a
\end{pmatrix} \leqslant \frac32 \begin{pmatrix}
1 & 0 \\ 0 & a
\end{pmatrix} \,,
\end{equation}
and thus, for $\xi\in [\lambda,\Lambda]$,
\[
\begin{pmatrix}
1 & c \\ c & a
\end{pmatrix}
 \begin{pmatrix}
0 & -1 \\ \xi  & \gamma
\end{pmatrix}   
 \geqslant  \  \begin{pmatrix}
\frac{\lambda}{2\gamma}  & 0 \\ 0 & \frac{a\gamma}{4}
\end{pmatrix}
\geqslant \frac23 \min \po \frac{\lambda}{2\gamma} , \frac{\gamma}{4}  \pf \begin{pmatrix}
1 & c \\ c & a
\end{pmatrix} = 
\frac{\lambda}{3\gamma} \begin{pmatrix}
1 & c \\ c & a
\end{pmatrix} 
\,,
\]
where we used that $\gamma^2 \geqslant 4\Lambda \geqslant 2\lambda$.
\end{proof}

\begin{proof}[Proof of Proposition~\ref{prop:NSconditionLangevin}]

The Jacobian matrix of the drift is 
\[J_b(x,y) \ = \ \begin{pmatrix}
0 & I_n \\ -\na^2 U(x) & -\gamma I_n
\end{pmatrix}\,.\]
If $\xi(x)$ is an eigenvalue of $\na^2U(x)$ for  $x\in\R^n$ then, by diagonalising  $\na^2 U(x)$, we see that $-\gamma/2\pm\sqrt{\gamma^2/4-\xi(x)}$ (possibly complex) are eigenvalues of $J_b(x,y)$ for all $y\in\R^n$. In particular, if there does not exist a $\lambda>0$ such that $\na^2 U(x) \geqslant \lambda I_n$ for all $x\in\R^n$ then the supremum over $x,y\in\R^n$ of the real parts of the eigenvalues of $J_b(x,y)$ is non-negative. As noticed at the beginning of Section~\ref{Subsec:resultat_courbure},  this proves that $\Contd$ cannot hold whatever $\rho>0$, $M\in\mathcal M_{2n,2n}^{sym>0}(\R)$. The first point of the proposition is proven.




We turn to the proof of the third point of the proposition, assuming that \eqref{eq:condUnabla2} holds with \nvv{$\sqrt{\Lambda}-\sqrt{\lambda}   <  \gamma$}. From Lemma~\ref{lem:condCNS}, we consider $a,c\in\R$  such that $c^2 <a$ and \eqref{eq:xi} holds. Let 
\[M \ = \ \begin{pmatrix}
I_n & c I_n \\ c I_n & a I_n
\end{pmatrix}\,, \qquad N_n(\xi) \ = \  \begin{pmatrix}
0 & -I_n \\ \xi I_n  & \gamma I_n
\end{pmatrix}\,.
\]
Then, $M$  and $M N_n(\xi)  $ are positive definite for $\xi \in \{\lambda,\Lambda\}$. In particular, there exists $\rho>0$ such that $MN_n(\xi) \geqslant \rho M$ for $\xi\in\{\lambda,\Lambda\}$. Let us prove that $MJ_{b}(x,y) \leqslant -\rho M$ for all $(x,y)\in\R^{2n}$, which will conclude the proof. Fix $(x,y)\in\R^{2n}$. Let $\mathcal O'$ be an orthonormal $n\times n$ matrix such that $(\mathcal O')^T \na^2 U(x) \mathcal O'$ is diagonal, and
\[\mathcal O \ = \ \begin{pmatrix}
\mathcal O' & 0 \\ 0 & \mathcal O'
\end{pmatrix}\,.\]
Then, $MJ_{b}(x,y) \leqslant -\rho M$ if and only if $\mathcal O^T M \mathcal O\mathcal O^TJ_{b}(x,y)\mathcal O \leqslant -\rho \mathcal O^T M \mathcal O$. For $i\in\cco 1,n\ccf$, let $R_i \in \mathcal M_{2,2n}(\R)$ be the matrix associated to $(x,y) \mapsto (x_i,y_i)$. Then,
\[\mathcal O^T M \mathcal O \ = \ M \ = \ \sum_{i=1}^n R_i^T \begin{pmatrix}
1 & c \\ c & a
\end{pmatrix}R_i\,,\qquad \mathcal O^T J_b(x,y)\mathcal O  \ =\  -\sum_{i=1}^n R_i^T N_1(\xi_i) R_i\]
where $(\xi_i)_{i\in\cco 1,n\ccf}$ are the eigenvalues of $\nabla^2 U(x)$. Since \eqref{eq:condUnabla2} holds, for all $i\in\cco 1,n\ccf$ there exists $p_i\in[0,1]$ such that $\xi_i = p_i \lambda + (1-p_i)\Lambda$, and thus
\[\sum_{i=1}^n R_i^T N_1(\xi_i) R_i \ = \ \sum_{i=1}^n R_i^T \po p_i N_1(\lambda) + (1-p_i)N_1(\Lambda)\pf R_i\ \geqslant \ \rho M\,.\]

The proof of the fourth point of the proposition is the same, using the second part of Lemma~\ref{lem:condCNS} and \eqref{eq:eqnorm}.

Finally, we prove the second point of the proposition. Working by contraposition, we suppose that $\Contd$ holds for some $\rho>0$ and $M\in\mathcal M_{2n,2n}^{sym>0}(\R)$ for a potential $U\in\mathcal C^2(\R^n)$ such that \eqref{eq:Usuffitpas} holds for some $x,x'\in\R^n$. For $i\in\cco 1,n\ccf$, notice that $R_i J_b(x,0)=N_1(\lambda) R_i$ and $R_i J_b(x',0)=N_1(\Lambda) R_i$. Applying  $\Contd$  at the  points $z=(x,0)$ and $z=(x',0)$ and denoting, for $i\in\cco 1,n\ccf$, $M_i = R_i M R_i^T$, which is a $2\times 2$ symmetric positive definite matrix, we see that necessarily $M_i N_1(\xi) \geqslant \rho M_i$ for both $\xi\in\{\lambda,\Lambda\}$, which from Lemma~\ref{lem:condCNS} proves that \nvv{$\sqrt{\Lambda}-\sqrt{\lambda}   <  \gamma$}.
\end{proof}

\subsection{Chains of oscillators}\label{Sec:chains}

Let us apply Proposition~\ref{prop:NSconditionLangevin} for potentials of the form \eqref{eq:pot_chains}. We start with the pinned case with a convex pinning potential, which is less interesting but simpler than the unpinned case.

\begin{prop}\label{prop:out-of-eq1}
Let $V,F\in \mathcal C^\infty(\R^p)$ with bounded Hessians and $F$ even.
Assume that $\na^2 F \geqslant -\kappa I_p$ for some $\kappa\geqslant 0$ and that $\na^2 V \geqslant (\lambda + 4\kappa)I_p$ for some $\lambda>0$.  Then for all $N\in\N$ the potential $U\in\mathcal{C}^\infty(\R^{pN})$ given by \eqref{eq:pot_chains} satisfies \eqref{eq:condUnabla2} with this value of $\lambda$ and with $\Lambda = \|\na^2 V\|_\infty+4\|\na^2 F\|_\infty$.

As a consequence, if $\gamma \geqslant 2\sqrt{\Lambda}$, for all $N\in\N$ the drift of the process \eqref{eq:chains} satisfies  $\Contd$ with $\rho = \lambda/(3\gamma)$ and a matrix $M$
which satisfies $\min(1,\Lambda^{-1})/2   \leqslant M \leqslant 3/2 \max(1,\Lambda^{-1}) $.
\end{prop}
\begin{proof}
For $u,x\in \R^{pN} $,
\[u\cdot \na^2 U(x) u = \sum_{i=1}^N u_i \na^2 V(x_i) u_i +\sum_{i=1}^{N-1} (u_i-u_{i+1})\cdot \na^2 F(x_i-x_{i+1})(u_i-u_{i+1}) \,,\]
so that the result follows from bounding $|u_i-u_{i+1}|^2 \leqslant 2|u_i|^2+2|u_{i+1}|^2$ for all $i\in\cco 1,N-1\ccf$ and applying Proposition~\ref{prop:NSconditionLangevin}. 
\end{proof}

The main point of this result is that the condition $\gamma \geqslant2\sqrt{\Lambda}$,  $\rho$ and the condition number of $M$ are independent from $N$. In particular, applying Theorem~\ref{thm:contraction} and  using the equivalence of $|\cdot|$ and $\|\cdot\|_M$, along the dynamics \eqref{eq:pot_chains}, the standard Euclidean Wasserstein distances are contracted in a time $t$  by a factor $3 \max(\Lambda,\Lambda^{-1}) e^{-\rho t}$. Applying Proposition~\ref{prop:concentration}, the process admits a unique invariant measure which satisfies a log-Sobolev inequality, with a constant independent from $N$ provided the temperatures $T_i$'s are bounded uniformly in $N$.

\bigskip

Now we consider the unpinned case, i.e. we suppose that $V=0$. It is clear that even if $F$ is strongly convex then $U$ given by \eqref{eq:pot_chains}  cannot be strongly convex (and thus, from Proposition~\ref{prop:NSconditionLangevin}, a contraction cannot hold) due to the invariance by the translation of the center of mass. More precisely, $U(x+z)=U(x)$ for all $z$ in the kernel of the orthogonal projector $\Pi$ given by $\Pi x=(x_1-\bar x,\dots,x_N-\bar x)$ with $\bar x=N^{-1}\sum_{i=1}^N x_i$ and the decomposition $x=(x_1,\dots,x_N)\in (\R^p)^N$. In fact in the unpinned case we are not interested in the full process $(X,Y)=(X_i,Y_i)_{i\in\cco 1,N\ccf}$ but in the centered process seen from its center of mass, $(\Pi X,\Pi Y)$, which is a process on the image of $\Pi$.  
The simplest way to enter the framework of the rest of our work is to consider a $p(N-1)\times pN$ matrix $Q$ with $\mathrm{Ker}Q=\mathrm{Ker}\Pi$ and such that the restriction of $Q$ on the image of $\Pi$ is an orthonormal isometry. Notice that then $Q^T$ is the $pN\times p(N-1)$ matrix such that the restriction of $Q^TQ$ to $\mathrm{Im}\Pi$ is the identity.  Set $(X^c,Y^c)=(Q X,Q Y)$, which is now a process on $\R^{p(N-1)}$, isometric to $(\Pi X,\Pi Y)$, so that all results of Wasserstein contraction, invariant measure, functional inequalities etc. immediately transfer from $(X^c,Y^c)$ to $(\Pi X,\Pi Y)$. From now on we focus on $(X^c,Y^c)$, which solves
 \begin{equation}\label{eq:chains_centered}
\left\{\begin{array}{rcl}
\dd X_{t}^c & = &  Y_{t}^c \dd t \\
\dd Y_{t}^c  & = & -   Q \na U(Q^T X_{t}^c)  \dd t     - \gamma Y_{t}^c\dd t + \tilde \Sigma \dd \tilde W_{t}
\end{array}
\right.
\end{equation}
where  $\tilde \Sigma$ is a constant matrix  and $\tilde W$ is a standard Brownian motion on $\R^{p(N-1)}$. Since $Q\na U(Q^T x) = \na U^c(x)$ with $U^c(x)=U(Q^T x)$, we get that the drift of $(X^c,Y^c)$ is of the form \eqref{eq:bsuffitpas} but with $U$ replaced by $U^c$.

%
%

\begin{prop}\label{prop:out-of-eq2}
Let $F\in \mathcal C^\infty(\R^p)$ be an even potential with bounded Hessian. Assume that $\na^2 F \geqslant \kappa I_p$ for some $\kappa>0$. Let $N\geqslant 2$,   $U$ be given by \eqref{eq:pot_chains} with $V=0$, and $U^c$ be defined as above with some matrix $Q$. Then, $U^c$ satisfies \eqref{eq:condUnabla2} with $\lambda = \kappa/N^2$ and $\Lambda =  4\|\na^2 F\|_\infty$.

As a consequence, if $\gamma \geqslant 2\sqrt{\Lambda}$, for all $N\in\N$ the drift of the process \eqref{eq:chains_centered} satisfies  $\Contd$ with $\rho = \kappa/(3\gamma N^2)$ and a matrix $M$
which satisfies $\min(1,\Lambda^{-1})/2   \leqslant M \leqslant 3/2 \max(1,\Lambda^{-1}) $.
\end{prop}
\begin{proof}
Since $\na^2 U^c (x) = Q\na^2 U(Q^T x) Q^T$ and $Q^T$ is an isometry from $\R^{p(N-1)}$ to $\mathrm{Im}\Pi = \{u\in\R^{pN},\ \bar u =0\}$, this is equivalent to prove that for all $u,x \in \R^{pN}$ with $\bar u = 0$,
\[\frac{\kappa}{N^2}  |u|^2 \leqslant u\cdot \na^2 U(x) u \leqslant 4\|\na^2 F\|_\infty |u|^2 \,.\]
Since
\[u\cdot \na^2 U(x) u   \ = \ \sum_{i=1}^{N-1} (u_i-u_{i+1})\cdot \na^2 F(x_i-x_{i+1}) (u_i-u_{i+1})\,, \]
the upper bound on $\na^2 U^c$ is similar to the proof of Proposition~\ref{prop:out-of-eq1} and the lower bound follows from the lower bound on $\na^2 F$ and the spectral gap of the discrete Laplacian on $\cco 1,N\ccf$, see  \cite[Proposition 17]{MoiGamma}.
\end{proof}

As in the pinned case, the condition $\gamma \geqslant 2\sqrt{\Lambda}$ on the friction and the condition number of $M$ does not depend on $N$, however now the contraction rate $\lambda$ is of order $1/N^2$. This is sharp as can be seen on the Gaussian case with $F(x)=|x|^2$. Using the equivalence of $|\cdot|$ and $\|\cdot\|_M$ yields a contraction of the standard Wasserstein distances, with an additional factor $3 \max(\Lambda,\Lambda^{-1})$.

In a more general case \eqref{eq:pot_graph}, the results and the proof are the same but the spectral gap of the Laplacian on the discrete line has to be replaced by the spectral gap of the Laplacian of the graph of interactions, as in \cite{MoiGamma} in the overdamped case.

\section{Application to generalised Langevin diffusions}\label{Sec:LangevinGeneralisee}

\subsection{Definition of the processes}\label{SubSec:Generalized}

\nvv{The objective of Section~\ref{Sec:LangevinGeneralisee}} is to extend the contraction results of  \cite{BolleyGuillinMalrieu,Chatterji1,Dalalyan} to the case of the generalised Langevin diffusions, in the sense of the Markov processes $Z=(X,Y)$ on $\R^n\times \R^p$ that solves 
\begin{equation}\label{eq:EDS}
\left\{\begin{array}{rcl}
\dd X_t & = & A Y_t \dd t \\
\dd Y_t & = & - A^T\na U(X_t) \dd t - B Y_t \dd t + \Sigma \dd W_t
\end{array}
\right.
\end{equation}
with $A\in\mathcal M_{n,p}(\R)$ (denoting by $\mathcal M_{k,\ell}(\R)$ the set of real matrices with $k$ lines and $\ell$ columns),  $B,\Sigma\in\mathcal M_{p,p}(\R)$, $U\in\mathcal C^2(\R^n)$, $W$ is a standard Brownian motion in  dimension $p$ and $H^T$ stands for the transpose of a matrix  $H$. 


Let us emphasize that the goal is not simply to quantify the convergence toward equilibrium (for instance in the sense of the   $\chi^2$ distance or of the relative entropy) of the process in a general case (which, for generalised Langevin diffusions, has been studied under some conditions in \cite{Ottobre,Pavliotis,PavliotisStoltz}, and for the classical Langevin is the topic of an abundant literature) but to establish a much stronger contraction property  under some restrictive conditions (although it turns out that the conditions on  $B$ that we will consider are in fact more general than in \cite{Ottobre,Pavliotis,PavliotisStoltz}). In particular, we  won't try to prove that some generalised Langevin diffusions \eqref{eq:EDS} converge  to equilibrium faster than the usual Langevin diffusion, or even than the overdamped Langevin diffusion. Moreover, we will only consider the continuous-time process, although the case of a numerical scheme based on \eqref{eq:EDS} can be treated in a similar way, as in \cite{MoiLangevinCinetique} for the usual Langevin diffusion.

As already mentioned in the Introduction, the  diffusion matrix $\Sigma$ does not play any role on the contraction property that is studied here. However, if  
\begin{equation}\label{eq:fluctuation-dissipation}
B + B^T \ = \  \Sigma \Sigma^T\,,
\end{equation}
\nvv{which is called a fluctuation-dissipation relation (as it relates the random fluctuation intensity with the dissipative friction, see \cite[Section 15.2.2]{Tuckerman})}
 then, assuming that   $\Sigma$ is symmetric positive  (which we can always do without loss of generality, upon replacing $\Sigma$ by the symmetric square-root of $\Sigma^T\Sigma$, which doesn't change the law of the process \eqref{eq:EDS}), it is completely determined by $B$, and more precisely by the symmetric part of $B$. The condition~\eqref{eq:fluctuation-dissipation} is equivalent to the fact that  the Gaussian part of the diffusion, namely
\begin{equation}\label{eq:partieOU}
\dd Y_t \ = \    - B Y_t + \Sigma \dd W_t\,,
\end{equation}
admits the standard Gaussian measure on $\R^p$ as an invariant measure.

Our goal is to work in a framework that allows for cases where this Gaussian part is itself degenerated, in the sense that the symmetric part of $B$ is not positive definite (the positive definite case would correspond, when the fluctuation-dissipation condition  \eqref{eq:fluctuation-dissipation} is satisfied, to the case where $\Sigma$ is invertible, and thus the process \eqref{eq:partieOU} is elliptic). Indeed, the case where the symmetric part of  $B$ is positive definite is not so different from the classical kinetic \nvv{Langevin diffusion \eqref{eq:Langevin_cinetique}}  (\nvv{which corresponds to} $p=n$ and $A=B=I_n$). We will keep in mind the following example:  $p=2n$,  
 \[A = \begin{pmatrix}
 I_n & 0
 \end{pmatrix}
 \qquad \text{and}\qquad
 B = \begin{pmatrix}
 0 & -I_n \\ I_n &  I_n
 \end{pmatrix}\,.\]
 In other words, in this case (and with the condition \eqref{eq:fluctuation-dissipation}), the  Gaussian part \eqref{eq:partieOU} is itself a kinetic Langevin diffusion. Decomposing  $Y=(Y_1,Y_2)$, the corresponding generalised Langevin diffusion  thus solves
 \begin{equation*} 
\left\{\begin{array}{rcl}
\dd X_t & = & Y_{1,t} \dd t \\
\dd Y_{1,t} & = & -  \na U(X_t)\dd t + Y_{2,t}\dd t \\
\dd Y_{2,t} &  = & - Y_{1,t} \dd t - Y_{2,t}\dd t  + \dd \tilde W_t
\end{array}
\right.
\end{equation*}
with $\tilde W$  standard Brownian motion of dimension $n$. This may be called an order 3 Langevin diffusion. More generally, for $K\geqslant 1$, we can consider the $K+1$ order Langevin diffusion, given by   $p=Kn$, 
\begin{equation}\label{eq:ABgeneralK}
A = \begin{pmatrix}
I_n & 0 & \dots   & 0
\end{pmatrix}\qquad \text{and}\qquad
B = \begin{pmatrix}
0  & -I_n & 0 & \dots & 0 \\
 I_n & 0 & -I_n & \ddots & \vdots \\
0 & \ddots & \ddots & \ddots &   0 \\
\vdots & \ddots &   I_n & 0 & - I_n\\
0 &\dots  & 0 &  I_n &   I_n
\end{pmatrix}\,.
\end{equation}

In this case, still with the condition \eqref{eq:fluctuation-dissipation}, the corresponding diffusion is highly degenerated: a system of dimension $(K+1)n$ is driven by a random noise of dimension $n$ and, as this noise is only propagated between adjacent coordinates  (i.e. from $Y_{i+1}$ to $Y_i$), Hörmander brackets of order $K$ have to be considered in order to span the full space    \cite{Hormander}.

The form of the   $K+1$ order Langevin diffusion is reminiscent of the diffusions designed in \cite{MonmarcheGuillin}. That article was concerned with the following question: for a given Gaussian target law, which generalised Ornstein-Uhlenbeck process converges fastest to this equilibrium? Under a normalisation condition (more precisely: the trace of the diffusion matrix being fixed), the answer is given by some diffusions for which the rank of the diffusion matrix is 1 (in other words there is noise directly only on one  variable). Of course, the situation is different to ours, since \cite{MonmarcheGuillin} worked within a given space, while the higher order Langevin diffusions are defined by adding new variables to the initial space  $\R^n$.

\subsection{Some motivations}\label{SubSec:LangevinGeneral}

The (usual) Langevin diffusion appears, among other contexts, in questions of \emph{effective dynamics}: given a high-dimensional system following the Hamiltonian dynamics, under some conditions, the trajectory of one coordinate approximately follows a Langevin diffusion, see  \cite[Chapter 15]{Tuckerman}. The friction/dissipation part corresponds to the interaction between the low dimensional system and its environment. However, in many cases this (Markovian) approximation is not able to reproduce some dynamical properties of the motion. The Mori-Zwanzig formalism \cite{Zwanzig,Mori} can lead to models with a non-Markovian friction (which corresponds to some memory effect of the environment in reaction to the influence of the small system), namely generalised Langevin equations of the form
\[\ddot{x_t} \ = \  -\na U(x_t)   - \int_0^t \Gamma(t-s) \dot{x}_s \dd s  + F_t\]
with a memory kernel $\Gamma:\R_+ \rightarrow \mathcal M_{n,n}(\R)$ and a Gaussian process $F$ whose covarince function is given, under the fluctuation-dissipation condition, by $\mathbb E(F_s F_t) = \Gamma(t-s)$ (the classical case corresponding to a Dirac mass at $0$ for $\Gamma$ and a white noise for $F$). In the particular case where $\Gamma(t) = \lambda^T \exp(tC) \lambda$ for some $\lambda \in \mathcal M_{p,n}(\R)$ and $C\in\mathcal M_{p,p}(\R)$, the memory can be decomposed in a finite number of variables so that the diffusion with memory is equivalent to an enlarged Markovian system 
\begin{equation*} 
\left\{\begin{array}{rcl}
\dd X_t & = & V_t \dd t \\
\dd V_t & = & -  \na U(X_t) \dd t + \lambda^T  S_t \dd t\\
\dd S_t & = & -\lambda V_t - C S_t \dd t + \Sigma \dd W_t\,,
\end{array}
\right.
\end{equation*}
with $\Sigma = (C+C^T)^{1/2}$. The log-density of the invariant measure is then $U(x)+|v|^2/2+|s|^2/2$. We refer the reader to the recent works  \cite{GeneralisedLangevin,GeneralisedLangevin2} and references within for more details on this matter.

Note that the form  \eqref{eq:EDS} is slightly more general, in particular it does not require the auxiliary Gaussian variable $Y$ to be decomposed as a velocity $V$ and a memory $S$ which are independent at equilibrium. This slightly more general form is motivated by MCMC methods. The context is then the following: the goal is to estimate some expectations with respect to a law $\pi(\dd x) \propto \exp(-U(x))\dd x$ on $\R^n$ for some $U$. To do so, one can design a Markov process  $(X,Y)$ on an extended space $\R^n\times\R^p$ ergodic with respect to a law $\pi \otimes \nu$, for some marginal $\nu$ to be chosen by the user. One of the objectives  is to design in that way processes that converge towards equilibrium faster than classical samplers (Metropolis-Hastings walk, overdamped Langevin diffusion\dots) that are typically reversible. Adding auxiliary variables gives systematic ways to design non-reversible samplers. It can also be understood as a way to create non-Markovian processes  $(X_t)_{t\geqslant 0}$ that are still ergodic with respect to the measure $\pi$. Indeed, the Markov property, useful in order to have a correct definition of the algorithm and of ergodicity, becomes a drawback concerning the efficiency of the exploration: an explorer without any memory always comes back to the same places without learning from what it has already seen (in particular, at each step, a reversible process has always, by design, a non-zero probability to go back immediately to where it was at the previous step, which after immobility is the second worst way to explore), while convergence toward equilibrium is related to the speed at which a statistically representative sample of states has been explored. A wide class of such lifted MCMC methods is given by kinetic samplers, namely methods for which the auxiliary variable is the velocity $V=\dd X/\dd t$. This case historically traces back at least to the molecular dynamics of Adler and Wrainwright \cite{AlderWainwright} based on Hamiltonian dynamics, leading to HMC methods \cite{HOROWITZ,Neal} (for \emph{Hamiltonian} or \emph{Hybrid Monte Carlo}). The Langevin diffusion belongs to this family, together with the more recent class of velocity jump process (for which the velocity $V$ is piecewise constant). We refer to \cite{Paulin,MoiLangevinCinetique} for more details and references on this (still) very vivid topic.

From the viewpoint of lifted MCMC, equation \eqref{eq:EDS} results from the following choices: the auxiliary marginal $\nu$ is a Gaussian distribution (and thus up to a linear change of variable we can assume it is a standard Gaussian) and, except for a term $-\na U$ which is also present in many other samplers,  the evolution is given by an Ornstein-Uhlenbeck diffusion, namely the drift is linear and the diffusion matrix is constant. The goal is to keep the simulation of the process simple enough: the auxiliary variable being added purposely in order to improve the algorithm, the point is not to increase significantly the numerical cost of the simulation (which is still concentrated here in the computation of $\na U$) or the sources of errors. In particular, like in the  case of the classical Langevin diffusion \cite{MoiLangevinCinetique}, a second order splitting integrator  can be used, separating the part \eqref{eq:partieOU} (which can be sampled exactly) from the (symplectic) Hamiltonian part $\dot x = A y$, $\dot y = -A^T\na U(x)$ with a Verlet scheme
  \begin{align*}
  y_{n+1/2} \ & = \ y_{n} - \frac12\delta A^T \na U(x_n)\\
  x_{n+1} \ & = \ x_{n} + \delta  A y_{n+1/2}\\
  y_{n+1} \ & = \ y_{n+1/2} - \frac12\delta A^T\na U(x_{n+1})\,,
  \end{align*}
where $\delta>0$ is the time-step (this part could be completed, like in HMC, with an accept/reject Metropolis step to correct the numerical bias due to time discretization). 

Besides, without claiming to compare the efficiency for MCMC of the generalised Langevin diffusions \eqref{eq:EDS} and HMC, let us make the following remark. One can consider the process  $(X_t,Y_t)_{t\geqslant 0}$ on $\R^n\times\R^p$ that follows the deterministic flow
\[\dot x_t \ = \ Ay_t\,,\qquad \dot y_t \ =\ -A^T \na U(x_t) \]  
and, at constant rate $\lambda$, get its auxiliary variable $Y$ partially refreshed by
 \[Y \ \leftarrow \ C Y + \po 1 - C^T C\pf^{1/2}G\]
for some $C\in\mathcal M_{p,p}(\R)$ with $|C|<1$, where $G$ is a standard Gaussian variable on $\R^p$ independent from the past. This process is linked to the generalised Langevin diffusion in the same way the randomized HMC of \cite{Paulin} (which corresponds to the case $p=n$ and $C=\alpha I_n$, $\alpha\in[0,1)$) is linked to the usual Langevin diffusion (in particular, by taking $C= e^{-\delta B}$ and $\lambda=1/\delta$,  the former converges toward the latter as $\delta\rightarrow 0$).  Actually, the process has the same advantage with respect to the classical HMC as the generalised Langevin diffusions with respect to the usual Langevin diffusion, namely they offer the possibility to introduce a memory effect that is richer than the one simply given by the velocity (but then, optimising the choice of $C$ to improve the efficiency of the process in term of MCMC algorithms is a difficult question, that we will not discuss here).  This (piecewise deterministic) process is by some aspects  quite close to the jump process studied in   Section~\ref{Subsec:PDMP} (in particular for a quadratic $U$). It is clear that there can never be an almost sure contraction, since the energy $U(x)+|y|^2/2$ is preserved along the deterministic flow, and for all $t\geqslant 0$ there is a non-zero probability that no jump has occurred during time $[0,t]$. Yet, by considering the synchronous coupling of two processes starting at two different initial conditions and taking the expectation of the square of their distances, it turns out that, similarly to \cite{Paulin} where  the same computations are obtained for the HMC process as in \cite{BolleyGuillinMalrieu,Chatterji1,Dalalyan}  for the Langevin diffusion, the computations are the same as for a generalised Langevin diffusion (with a matrix $B=-\lambda C$). This means that a contraction of the   $\mathcal W_{M,2}$ distance for some $M$ can be established under the same assumptions in both cases. Yet, in the diffusion case, by Theorem~\ref{thm:contraction}, this contraction is equivalent to the contraction of the  $\mathcal W_{M,\infty}$ distance, which does not hold for the piecewise deterministic process. 

In relation with this, we mention that the almost sure contraction for the discrete-time HMC (where refreshment of the velocity occur at deterministic times) in the case where $U$ is convex with a bounded Hessian has been established in  \cite{MangoubiSmith}.

\subsection{Contraction for generalised Langevin diffusions}\label{Sec:contractLangevinGeneralise}

Recall that the goal of this section is to give quantitative conditions under which  $\Contd$ holds for some  $\rho>0$ for the Markov process $Z=(X,Y)$ on $\R^n\times \R^p$ which solves
\begin{equation}\label{eq:EDS_Langevin_generalise_avec_gamma}
\left\{\begin{array}{rcl}
\dd X_t & = & A Y_t \dd t \\
\dd Y_t & = & - A^T\na U(X_t) \dd t - \gamma B Y_t \dd t + \sqrt {\gamma } \Sigma \dd W_t\,,
\end{array}
\right.
\end{equation}
where  $A\in\mathcal M_{n,p}(\R)$,  $B,\Sigma\in\mathcal M_{p,p}(\R)$,  $\gamma>0$, $U\in\mathcal C^\infty(\R^n)$ with all its derivatives of order $2$ or more bounded, and $W$ is a standard Brownian motion of dimension $p$. With respect to \eqref{eq:EDS}, we have added a scalar friction parameter $\gamma$. The goal is to have a simple way to state, for a given $B$, a condition of the form ``\emph{provided the friction is high enough, \dots}",  since we have seen in the previous section that such a restriction is already necessary in the case of the classical Langevin diffusion. For the same reason, we suppose that there exist  $\lambda,\Lambda>0$ such that \eqref{eq:condUnabla2} holds.  Notice that, in that case, assuming that $p=Kn$ for some $K\geqslant 1$ and that $B$ and $A$ are constituted of $n\times n$ blocks which are homotheties, then the proof of Proposition~\ref{prop:NSconditionLangevin} is straightforwardly adapted to the generalised case, and the question boils down to a linear algebra problem. However, in the following, we will not assume such a condition on $p$, $A$ and $B$.

It is also clear that, in order to get a contraction, the rank of $A$ necessarily has to be $n$ (and in particular $p\geqslant n$): otherwise, by considering some non-zero $x \in \mathrm{Ker}(A^T)$, we get that $(x,0)\in \mathrm{Ker}(J_b^T(z))$ for all $z\in\R^d$, so that $0$ is an eigenvalue of $J_b(z)$  and a contraction cannot hold. Another way to see it is that, for such an $x$, $\dd( x\cdot X_t) = 0$ for all $t\geqslant 0$ and thus, with a parallel coupling, $|x\cdot (X_{1,t}-X_{2,t})|= |x\cdot(X_{1,0}-X_{2,0})|$ for all $t\geqslant 0$.

Provided that the rank of $A$  is $n$, up to a change of variable $y\leftarrow Qy$ with an orthonormal transformation $Q$ that sends the kernel of $A$ to the $n-p$ last vectors of the canonical basis, we can suppose that  $A=(\tilde A,0,\dots,0)$ where $\tilde A\in\mathcal M_{n,n}(\R)$ is invertible. Then, the change $x\leftarrow \tilde A^{-1}x$ enables to chose $A=(I_n,0,\dots,0)$. The potential $U$ is then replaced by $x\mapsto U(\tilde Ax)$, which does not change the fact that a  condition  \eqref{eq:condUnabla2} is satisfied, up to a modification of the values of $\lambda,\Lambda$. As a consequence, in the following, without loss of generality, we assume that $A=(I_n,0,\dots,0)$. 

It is also natural to  require from  $B$  that the Gaussian process \eqref{eq:partieOU} is stable, in other words that the real parts of the eigenvalues of $B$ are all positive (for $p=n=1$  it is clearly a necessary condition to get a contraction). According to \cite[Lemma 2.11]{Arnold}, this is equivalent to saying that there exist $N\in\mathcal M_{p,p}^{sym>0}(\R)$ and $\kappa>0$ such that  $NB\geqslant \kappa N$.

\nv{
As a summary, up to now we have argued that, in order to establish a contraction for the generalised Langevin process, without loss of generality, it is necessary to consider the following  set of conditions: 
\begin{assu}\label{Hyp_simple}
There exist $\lambda,\Lambda,\kappa>0$ and  $N\in\mathcal M_{p,p}^{sym>0}(\R)$ such that $NB\geqslant  \kappa N$ and  \eqref{eq:condUnabla2} holds. Moreover, $p\geqslant n$ and $A=(I_n,0,\dots,0)$. 
\end{assu}
However, without further assumptions on the structure of $B$, it is difficult to give explicit bounds at this level of generality. Indeed, to motivate additional assumptions, let us begin the computations to see how far we go just with Assumption~\ref{Hyp_simple}.
\begin{lem}\label{Lem:calculpreliminaire}
Under Assumption~\ref{Hyp_simple}, let $a \in \mathcal M_{n,n}^{sym>0}(\R)$, $A^{-1}\in\mathcal M_{p,n}$ be a right inverse of $A$, $\theta_1,\theta_2,\alpha>0$, $\xi\in \R$,
\[M \ = \ \begin{pmatrix}
\gamma a & (A^{-1})^T \\ A^{-1} & \gamma \alpha N
\end{pmatrix}\]
and $b$ be the drift of the process~\eqref{eq:EDS_Langevin_generalise_avec_gamma}. Then, for all $(x,y)\in\R^{n+p}$,
\[M J_b(x,y) \ \leqslant \ \max_{s\in \{\lambda,\Lambda\}} \begin{pmatrix} -  s + \frac{(s-\xi)^2}{2\theta_1} + \frac{1}{2\theta_2} & 0\\ 0 & A^{-1}A + \gamma^2 \co \frac{\theta_1 \alpha^2}{2} N A^T A N+ \frac{\theta_2 }{2} C^TC-\alpha \kappa N \cf\end{pmatrix} \]
where $C= a A-(A^{-1})^T B- \xi \alpha A N $.
\end{lem}
\begin{proof}
To motivate the particular choice of $M$, let us first consider  $M$ of the general form
\[M \ = \ \begin{pmatrix}
a' & b \\ b^T & c
\end{pmatrix}\,,\qquad a' \in \mathcal M_{n,n}^{sym>0}(\R),\ b \in \mathcal M_{n,p}(\R),\ c \in \mathcal M_{p,p}^{sym>0}(\R)\,.\]
Then,
\[M J_b(x,y) \ = \ \begin{pmatrix}
- bA^T\na^2 U(x) & a' A-\gamma b B \\- cA^T\na^2 U(x) & b^T A - \gamma cB
\end{pmatrix}\,.\]
The diagonal terms lead us to choose on the one hand $b=\beta (A^{-1})^T $ with $\beta>0$ and $A^{-1}\in\mathcal M_{p,n}$ a right-inverse of $A$, and on the other hand $c=\alpha' N$ for some $\alpha' >0$. By homogeneity (i.e. the fact that multiplying $M$ by a positive constant has no effect on the contraction property we are interested in), we can chose $\beta=1$. Since we do not want to assume any particular relation between the eigenvectors of $\na^2 U(x)$ for $x\in\R^n$ and the constant matrices $B,N$, we are also naturally lead to bound, introducing parameters $\xi\in\R$ and $\theta_1>0$, the diagonal term involving $\na^2 U$ as 
\[\forall u\in\R^n,v\in\R^p,\qquad |v\cdot N A^T(\na^2 U(x)-\xi I_n) u | \ \leqslant \ \frac{\theta_1}{2}|ANv|^2 + \frac{1}{2\theta_1} |(\na^2 U(x)-\xi I_n) u |^2\,,\]
in other words
\[\begin{pmatrix}
0 & 0\\- \alpha' NA^T\na^2 U(x) & 0
\end{pmatrix} \ \leqslant \ \begin{pmatrix}
 \frac{1}{2\theta_1}  (\na^2 U(x)-\xi I_n)^2 & 0 \\- \xi \alpha' N A^T  &\frac{\theta_1 (\alpha')^2}{2} N A^T A N 
\end{pmatrix}\,.\]
To sum up, with the choice $b=(A^{-1})^T$ and $c=\alpha' N$, we have obtained that, for all $\xi\in\R$ and $\theta_1>0$,
\begin{equation}\label{eq:matrices}
M J_b(x,y)  \leqslant   \begin{pmatrix}
-  \na^2 U(x) + \frac{1}{2\theta_1}  (\na^2 U(x)-\xi I_n)^2 & a' A-\gamma (A^{-1})^T B \\- \xi \alpha' N A^T  &\frac{\theta_1 (\alpha')^2}{2} N A^T A N+ A^{-1} A - \gamma \alpha' \kappa N
\end{pmatrix}.
\end{equation}
Denoting $f(s,\theta,\xi) = - s +(s-\xi)^2/(2\theta)$, which is convex in $s$, by diagonalizing $\nabla^2 U(x)$ we see that for all $x\in\R^n$
\[ -  \na^2 U(x) + \frac{1}{2\theta}  (\na^2 U(x)-\xi)^2 \ \leqslant \ \max_{s\in[\lambda,\Lambda]} f(s,\theta,\xi) \ = \ \max_{s\in\{\lambda,\Lambda\}} f(s,\theta,\xi)\,.\]
Since we are interested in the large $\gamma$ regime, we chose $a'=\gamma a$ and $\alpha'=\gamma \alpha$ to ensure some homogeneity in $\gamma$, so that the diagonal term of the quadratic form associated to the right hand side of \eqref{eq:matrices} is $u,v\mapsto \gamma u \cdot Cv$. We bound this term  as $\gamma u \cdot Cv \leqslant |u|^2 /(2\theta_2) + \gamma^2 \theta_2 |Cv|^2/2$ for some $\theta_2>0$, which concludes. 
\end{proof}

Lemma~\ref{Lem:calculpreliminaire} can serve as a preliminary result when trying to prove a contraction for a generalized Langevin process with some particular, explicit $B$. In the following, we add to Assumption~\ref{Hyp_simple} another condition in order to state  a simple result for a class of matrices $B$ which contains the examples given by \eqref{eq:ABgeneralK} and more generally many cases where $B$ is constituted of $n\times n$ homothety blocks, which is a natural practical case for lifted MCMC methods. 

First, when $p>n$ and $A=(I_n,0,\dots,0)$, it is natural to consider the decomposition
\begin{equation}\label{eq:B22}
B \ = \ \begin{pmatrix}
B_{11} & B_{12} \\ B_{21} & B_{22}
\end{pmatrix} 
\end{equation}
with $B_{11},B_{12},B_{21}$ and $B_{22}$ of respective sizes $n\times n$, $n\times(p-n)$, $(p-n)\times n$ and $(p-n)\times(p-n)$. We notice that a contraction for some $\rho>0$ for \eqref{eq:EDS_Langevin_generalise_avec_gamma} necessarily requires   $B_{22}$ to be invertible. Indeed, otherwise, considering a non-zero element $u\in\R^{p-n}$ in the kernel of $B_{22}^T$, then $(-B_{21}^Tu,0,u)$ is in the kernel of $J_b(x,y)^T$ for all $(x,y)\in\R^{n,p}$, so that $0$ is in the spectrum of $J_b(x,v)$ which, as we already noticed, prevents a contraction with $\rho>0$. Our additional condition is the following:

\begin{assu}\label{Hyp_simple2}
\begin{itemize}\

\item If $p>n$, considering the decomposition \eqref{eq:B22}, we suppose that $B_{22}$ is invertible and that
\begin{equation}\label{eq:matriceC}
E \ :=\ B_{11} - B_{12} B_{22}^{-1} B_{21}
\end{equation}
is a positive definite symmetric matrix. We write $D=B_{12} B_{22}^{-1}$.
\item If $p=n$, we assume that $B$ is symmetric  positive definite, and we write $D=0$ and $E=B$.
\end{itemize}  
\end{assu}

Under Assumption~\ref{Hyp_simple} and \ref{Hyp_simple2} we will denote by $h_i>0$ for $i\in\cco 1,5\ccf$ some constants such that
\begin{equation}\label{eq:defh1h2}
\begin{array}{cc}
NA^TAN \ \leqslant \ h_1 N\,,\qquad &
 \frac1{h_2} I_n \leqslant E \leqslant h_3 I_n\,,\\
 & \\
  \begin{pmatrix}
I_n &  - D  \\  0  & 0 
\end{pmatrix}\  \leqslant \ h_4 N\,,\qquad &
  \begin{pmatrix}
I_n &  -D\\  -D^T & 0
\end{pmatrix} \leqslant h_5 N\,, 
\end{array}
\end{equation}
(with simply $I_n\leqslant h_4 N$ and $h_5=h_4$ in the case $p=n$) whose existence is ensured by the fact that $N$ and $E$ are positive definite. In particular one can chose $h_1 = |AN^{1/2}|^2$, $h_2 = |E^{-1}|$ and $h_3=|E|$.

Let us show  that Assumption~\ref{Hyp_simple2} is satisfied in the case where   $B$ is given by \eqref{eq:ABgeneralK} for some $K\geqslant 2$. We check that, for $D=-(I_n,\dots,I_n)\in\mathcal M_{n,p-n}$, $D B_{22} = B_{12}$,  
 from which $E=I_n$. This computation is unchanged if we add to $B_{22}$ any matrix whose coefficients sum to zero on each column.  More generally, notice that, if $B_{22}$ is invertible and $B$ is in the form of   $n\times n$ homothety blocks, then $E$ is an homothety, in other words the problem boils down to the case $n=1$, the symmetric character is trivial and the assumption on $E$ only concerns its sign. In particular, all cases  considered in \cite{Ottobre,Pavliotis,PavliotisStoltz} are covered by Assumption~\ref{Hyp_simple2}. Remark\footnote{This was indicated to the author by Gabriel Stoltz.} that $E$  is a Schur complement of $B$ (see \cite{Schur} for details), namely decomposing 
 \[B^{-1} \ = \ \begin{pmatrix}
 (B^{-1})_{11} &  (B^{-1})_{12}  \\ (B^{-1})_{21}  & (B^{-1})_{22} 
 \end{pmatrix}
 \]
 similarly to $B$ then $E^{-1}= (B^{-1})_{11} $ (see e.g. \cite[Equation 3.4]{Schur}).

\begin{thm}\label{thm:contractLangevin}
Under  Assumptions~\ref{Hyp_simple} and \ref{Hyp_simple2}, set
\begin{equation}\label{eq:gamma0}
\gamma_0 \ = \ 2 \sqrt{\frac{h_1\Lambda}{\kappa}}\max \po \sqrt{ h_2h_5  },\sqrt{\frac{h_4}{\kappa}} \pf\,.
\end{equation}
If $\gamma \geqslant \gamma_0$, then the drift of the process \eqref{eq:EDS_Langevin_generalise_avec_gamma} satisfies $\Contd$ with 
\begin{equation}\label{eq:rho}
\rho \ = \  \min \po\frac{\lambda}{3h_3\gamma}, \frac{\gamma \kappa }{6}\pf
\end{equation}
where $M$ is an explicit matrix such that
\[\frac{1}2 \begin{pmatrix}
 E & 0 \\ 0 & \frac{ \kappa}{\Lambda h_1} N
\end{pmatrix} \ \leqslant \ M \ \leqslant \ \frac{3}2 \begin{pmatrix}
 E & 0 \\ 0 & \frac{\kappa}{\Lambda h_1} N
\end{pmatrix} \,.
\]
\end{thm}
\noindent\textbf{Remarks:} 
\begin{itemize}
\item In the case of the usual Langevin diffusion (for which $E=N=A=B=I_n$), this proposition shows that there is a contraction for $\gamma\geqslant 2\sqrt\Lambda$. Considering that we have proven in Proposition~\ref{prop:NSconditionLangevin} that Assumption~\ref{Hyp_simple} could be satisfied while no contraction holds if, in particular, $\Lambda-\lambda   \geqslant   \gamma\po \sqrt{\Lambda}+\sqrt{\lambda}\pf$, we see that Theorem~\eqref{thm:contractLangevin} is relatively sharp, at least away from the regime  $\lambda \simeq \Lambda$. Incidentally, in order to deal with the   situation $\lambda \simeq \Lambda$, one can show that, for any values of $\Lambda$ and $\gamma$, there is a contraction if $\lambda$ is close enough to $\Lambda$. Indeed, it is possible to prove that    $\Contd$ holds for a given couple $(M,\rho)$ with $\rho>0$ in the Gaussian case $U(x) = \Lambda|x|^2/2$, and then to treat the case $\lambda \simeq \Lambda$ as a perturbation, namely
\begin{align*}
J_b(x,y) M \ & \leqslant \ -\rho M + \begin{pmatrix}
0 & 0 \\
-A^T (\na^2U(x) - \Lambda) & 0
\end{pmatrix} M   \\
& \leqslant \  - \po\rho  -  (\Lambda-\lambda)|M^{1/2}||M^{-1/2}|\pf M.
\end{align*}
We will not detail this case.
\item As $\gamma \rightarrow +\infty$, $\rho$ scales as $1/\gamma$, which is well-known in the classical Langevin case and is related to the fact that time has to be accelerated by a factor $\gamma$ for the process to converge to a non-degenerate overdamped limit in this regime. Moreover, as already mentioned in the standard Langevin case, $\lambda$ and $B$ (hence $\kappa$) being fixed, as $\gamma\rightarrow +\infty$ and time is accelerated by a factor $\gamma$, we see that $\Lambda$ may take larger and larger values while $\rho \gamma$ converges to $\lambda /(3h_3)$, which is consistent with the fact that, for the overdamped process, there is a contraction with $\rho = \lambda$ as soon as $U$ is convex, without the constraint that $\na^2 U$ has to be bounded.
\end{itemize}

\begin{proof}
The goal is to chose, in Lemma~\ref{Lem:calculpreliminaire}, the matrix $a$ so that $C=0$. Indeed, afterwards, it is clear that the bound on $MJ_b$ can be made negative by choosing, first, $\theta_1$ sufficiently high, then $\alpha$ sufficiently small, then $\gamma$ sufficiently large (which will also ensure that $M$ is positive definite). We simply chose $\xi=0$,  consider a right inverse of $A$ of the form $A^{-1} = (I_n\  -D)^T$ with $D$ as in Assumption~\ref{Hyp_simple2} (simply $A^{-1}=A=I_n$ if $p=n$) and set $a= (A^{-1})^T B A^{-1}$. With this choice,
\begin{align*}
C \ = \ a A-(A^{-1})^T B \ & = \ (A^{-1})^T B \po A^{-1} A - I_{p}\pf \\
& = \ \begin{pmatrix}
 ( D B_{22}-B_{12}) D^T \ &\  D B_{22}-B_{12}) 
\end{pmatrix} \ = \ 0
\end{align*}
since $DB_{22}=B_{12}$ (when $p=n$, $C=B-B=0$). Moreover, 
\[a \ = \ B_{11} - D B_{21} +\po DB_{22}-B_{12}\pf D^T \ = \ E\,,\]
which under Assumption~\ref{Hyp_simple2} is positive definite (also in the case $p=n$ with $a=B$). As a consequence, applying Lemma~\ref{Lem:calculpreliminaire} with $A^{-1} = (I_n\  -D)^T$, $a=E$, $\xi=0$, $\theta_2$ arbitrarily large (since $C=0$), we get that for all $(x,y)\in\R^{n+p}$ and all $\theta_1>0$
\[M J_b(x,y) \ \leqslant \ \max_{s\in \{\lambda,\Lambda\}} \begin{pmatrix} -  s \po 1 - \frac{s}{2\theta_1} \pf I_n & 0\\ 0 & A^{-1}A + \gamma^2 \co \frac{\theta_1 \alpha^2}{2} N A^T A N -\alpha \kappa N \cf\end{pmatrix} \,.\] 
Choosing $\theta_1 = \Lambda$, we get
\[M J_b(x,y) \ \leqslant \  \begin{pmatrix} - \frac{\lambda}{2} I_n & 0\\ 0 & \po h_4 + \gamma^2 \alpha \co \frac{\Lambda \alpha }{2} h_1 - \kappa\cf\pf N \end{pmatrix} \,.\] 
Choosing $\alpha = \kappa /(\Lambda h_1)$ and assuming that $\gamma^2 \geqslant \gamma_0^2 \geqslant 4h_4/(\kappa \alpha)$, this means 
\[M J_b(x,y) \ \leqslant \ - \begin{pmatrix} \frac{\lambda}{2} I_n & 0\\ 0 &    \frac{\gamma^2 \alpha \kappa}4   N \end{pmatrix} \ \leqslant \ - \frac32 \rho \begin{pmatrix} \gamma a & 0\\ 0 &     \gamma  \alpha   N \end{pmatrix} \]
with $\rho$ given by  \eqref{eq:rho}. Let us prove that, provided $\gamma$ is large enough, the matrix $M$ given in Lemma~\ref{Lem:calculpreliminaire} satisfies
\begin{equation}\label{eq:M}
\frac12 \begin{pmatrix} \gamma a & 0\\ 0 &     \gamma  \alpha   N \end{pmatrix} \ \leqslant \ M \ \leqslant \frac32  \begin{pmatrix} \gamma a & 0\\ 0 &     \gamma  \alpha   N \end{pmatrix}\,.
\end{equation}
This will concludes the proof (replacing $M$ by $M/\gamma$), since we will have obtained that $M$ is positive definite and that $M J_b(x,y) \leqslant - \rho M$ for all $(x,y)\in\R^{n+p}$. Now, for all $x\in\R^n$, $y\in\R^p$ and $\theta'>0$,
\[2| y\cdot A^{-1} x | \  \leqslant  \ \theta'|x|^2 +\frac1{\theta'} |(A^{-1})^Ty|^2 \ \leqslant \ \theta' h_2  x\cdot a x +\frac{h_5}{\theta' } y\cdot N y \,.\]
The inequalities \eqref{eq:M} thus hold if
\[2\theta' h_2 \ \leqslant \ \gamma \qquad \text{and}\qquad  \frac{2h_5}{\theta'} \leqslant  \gamma\alpha  = \frac{\gamma\kappa}{\Lambda h_1}\,.\]
These two constraints on $\gamma$ are the same when $(\theta')^2 = h_1h_5\Lambda/(h_2\kappa)$ and read in this case $\gamma \geqslant 2\sqrt{h_1h_2h_5 \Lambda/\kappa}$.

\end{proof}
}

\subsection{$L^2$ hypocoercivity}\label{Subsec:L2}

In this section, we consider the process \eqref{eq:EDS_Langevin_generalise_avec_gamma} in the case where the fluctuation-dissipation condition~\eqref{eq:fluctuation-dissipation} holds, and under Assumptions~\ref{Hyp_simple} and \ref{Hyp_simple2} with $\gamma\geqslant \gamma_0$ (given by Theorem~\ref{thm:contractLangevin}). The measure $\mu_\infty$ with density proportional to $\exp(-H(x,y))$ with $H(x,y)=U(x)+|y|^2/2$ is then the unique invariant measure of the process. Indeed,  its generator can be decomposed as  $\mathcal L = \mathcal L_{Ham} + \gamma \mathcal L_{B}$ with
\[\mathcal L_{Ham}f \ = \ (Ay)\cdot \na_xf - \po A^T \na U(x)\pf \cdot \na_yf\,,\qquad \mathcal L_{B} \ = \ -(By)\cdot\na_y f + \frac12 \na_y\cdot \po   \Sigma \Sigma^T   \na_y f \pf \,,\]
where $\na\cdot$ stands for the divergence operator. Then $\mathcal L_{Ham} H(x,y) = 0$ for all $(x,y)\in\R^n\times\R^p$, and thus any probability density that is a function of   $H$ is invariant for $\mathcal L_{Ham}$. Integrating by part, we see that $\nu\otimes \mathcal N(0,I_p)$ is invariant by $\mathcal  L_{B}$ for all $\nu \in \mathcal P(\R^n)$. Uniqueness is a consequence of the contraction, see Proposition~\ref{prop:concentration}.

The goal of this section is to establish the exponentially fast convergence of $\|P_t  - \mu_\infty \|_{L^2(\mu_\infty)}$ toward $0$ as $t\rightarrow +\infty$. The uniform convexity of $U$ (and thus of $H$) implies that  $\mu_\infty$ satisfies a Poincaré inequality with constant $1/\min(1,\lambda)$ (see \cite[Proposition 4.8.1]{BakryGentilLedoux}), so that
\begin{align*}
\| P_t f - \mu_\infty f\|_{L^2(\mu_\infty)}^2 \ & \leqslant \ \frac{|M|}{\min(1,\lambda)} \int_{\R^{n+p}} \|\na P_t f\|_{M^{-1}}^2 \dd \mu_\infty  \\
& \leqslant \ \frac{|M|}{\min(1,\lambda)} e^{-2\rho(t-s)}\int_{\R^{n+p}} \|\na P_s f\|_{M^{-1}}^2 \dd \mu_\infty
\end{align*}
for all $s\in[0,t]$, where  $M$ and $\rho$ are given by Theorem~\ref{thm:contractLangevin} and we used Theorem~\ref{thm:contraction}. In order to conclude, we could check that \cite[Theorem 9]{MoiGamma} applies in order to get that there exist (explicit)  $C>0$ and $k\in\N_*$ such that, for all $s>0$ and $f\in L^2(\mu_\infty)$,
\[\int_{\R^{n+p}} \|\na P_s f\|_{M^{-1}}^2 \dd \mu_\infty \ \leqslant \ \frac{C}{(1-e^{-s})^k} \|f -\mu_\infty f\|_{L^2(\mu_\infty)}^2\,.\]
This would give an estimate of the form
\[\| P_t - \mu_\infty\|_{L^2(\mu_\infty)} \ \leqslant \ C' e^{-\rho t}\,,\]
with the rate $\rho$ given by the almost sure convergence. However, for highly degenerated diffusions, \cite[Theorem 9]{MoiGamma} yields a constant $C$ which is possibly very bad, see for instance \cite{MonmarcheGuillin}. For this reason, we will rather follow the strategy used in \cite{Paulin}, which does not rely on a regularisation property, but rather on the fact that the dual   $\mathcal L^*$  of $\mathcal L$ in $L^2(\mu_\infty)$ has a structure that is close to $\mathcal L$.

 \begin{prop}\label{prop:L2}
Under Assumptions~\ref{Hyp_simple} and \ref{Hyp_simple2}, let $h_i'>0$ for $i\in\cco 1,5\ccf$  be such that \eqref{eq:defh1h2} are satisfied with $N$ and $B$  replaced by $N^{-1}$ and $B^T$, and let $\gamma_0'$ be given by \eqref{eq:gamma0} with the $h_i$ replaced by the $h_i'$. Suppose that $\gamma \geqslant \max(\gamma_0,\gamma_0')$. Then, for all $t\geqslant 0$,
 \[\|P_t -\mu_\infty\|_{L^2(\mu_\infty)} \ \leqslant \  \sqrt{3|N^{-1}||N|} e^{-\rho t}\,,\]
where $\rho$ is given by \eqref{eq:rho}.
 \end{prop}
 
 \noindent\textbf{Remark:} of course, this convergence in $L^2$ is weaker than the contraction of  Theorem~\ref{thm:contractLangevin}, in the sense that we expect it to hold for some rate $\rho>0$ and a prefactor in much a more general case, particularly without any restriction on $\gamma$ and with $U$ which is not convex but simply such that $\exp(-U)$ satisfies a Poincaré inequality, see for instance \cite{Pavliotis,Ottobre,PavliotisStoltz}.
 
 \begin{proof}
We compute $\mathcal L^*$. As a first order differential operator that leaves $\mu_\infty$ invariant, $\mathcal L_{Ham}$ is necessarily skew-symmetric ($\mathcal L_{Ham}^* = -\mathcal L_{Ham}$),  since, for $f,g$ $\mathcal C^\infty$ with compact support,
 \[0 \ = \ \int_{\R^{n+p}} \mathcal L_{Ham}(fg) \dd \mu _\infty \  = \ \int_{\R^{n+p}} f \mathcal L_{Ham} g \dd \mu _\infty + \int_{\R^{n+p}} g\mathcal L_{Ham}f \dd \mu _\infty\,.\]
 Integrating by parts, (see also \cite{MonmarcheGuillin}) we check  that $\mathcal L_B^* = \mathcal L_{B^T}$. In other words, $\mathcal L^*$ is the generator of the process $(X^*_t,Y^*_t)_{t\geqslant 0}$ that solves
 \begin{equation*} 
\left\{\begin{array}{rcl}
\dd X_t^* & = & -A Y_t^* \dd t \\
\dd Y_t^* & = &  A^T\na U(X_t^*) \dd t - \gamma B^T Y_t^* \dd t + \sqrt \gamma \Sigma \dd W_t\,,
\end{array}
\right.
\end{equation*}
so that $(X^*_t,-Y_t^*)_{t\geqslant 0}$ is a solution of  \eqref{eq:EDS_Langevin_generalise_avec_gamma} where we have simply substituted $B$ by $B^T$. Consider the associated semigroup $R_t = Q(P_t)^*Q$, where $Q$ is the operator $Qf(x,y)=f(x,-y)$. Let us show that Assumptions~\ref{Hyp_simple} and \ref{Hyp_simple2} are also satisfied in this  case. The fact that $NB \geqslant \kappa N$ is equivalent to $N^{-1} B^T \geqslant \kappa N^{-1}$. The conditions on $U$ and $A$ are unchanged. Finally, the matrix $C$ defined by \eqref{eq:matriceC} is replaced by $C^T$, which by assumption is equal to $C$ and is positive definite. As a  consequence, Assumptions~\ref{Hyp_simple} and \ref{Hyp_simple2} hold and we have simply replaced $N$ by $N^{-1}$. Theorem~\ref{thm:contractLangevin} thus applies to $(R_t)_{t\geqslant 0}$ with the $h_i$ replaced by the $h_i'$ and $M$ by $M'$ that satisfies
\[\frac{1}2 \begin{pmatrix}
 C & 0 \\ 0 & \frac{ \kappa}{\Lambda h_1'} N^{-1}
\end{pmatrix} \ \leqslant \ M' \ \leqslant \ \frac{3}2 \begin{pmatrix}
 C & 0 \\ 0 & \frac{\kappa}{\Lambda h_1'} N^{-1}
\end{pmatrix} \,.
\]
Set
\[K \ = \ \begin{pmatrix}
 C & 0 \\ 0 & \frac{\kappa}{\Lambda } I_p
\end{pmatrix} \,,\]
so that
\[ \frac{1}{2|N|h_1'}K\ \leqslant \ M' \ \leqslant \ \frac{3|N^{-1}|}{2h_1'}K\qquad\text{and}\qquad \frac{1}{2|N^{-1}|h_1}K\ \leqslant \ M \ \leqslant \ \frac{3|N|}{2h_1}K\,.\]
 Besides, the norm $\|\cdot\|_{K}$ is invariant by  $(x,y)\leftarrow (x,-y)$, so that $\mathcal W_{K,2}(\nu Q,\mu Q) = \mathcal W_{K,2}(\nu ,\mu )$ for all $\nu,\mu\in\mathcal P(\R^d)$. Using Theorem~\ref{thm:contraction}, for all $\nu,\mu\in\mathcal P(\R^d)$ and $t\geqslant 0$, 
 \begin{align*}
 \mathcal W_{K,2}^2\po \nu (P_t)^* P_t,\mu (P_t)^* P_t\pf  \ &\leqslant \ 2|N^{-1}|h_1 \mathcal W_{M,2}^2\po \nu (P_t)^* P_t,\mu (P_t)^* P_t\pf  \\
 &\leqslant \  2|N^{-1}|h_1 e^{-2\rho t}\mathcal W_{M,2}^2\po \nu (P_t)^*  ,\mu (P_t)^*  \pf  \\
 &\leqslant \  3|N^{-1}||N| e^{-2\rho t}\mathcal W_{K,2}^2\po \nu (P_t)^*  ,\mu (P_t)^*  \pf  \\ 
  &=\  3|N^{-1}||N| e^{-2\rho t} \mathcal W_{K,2}^2\po \nu QR_t ,\mu Q R_t \pf\\
    &\leqslant \  \eta_t^2 \mathcal W_{K,2}^2\po \nu  ,\mu  \pf 
 \end{align*}
 with $\eta_t = 3|N^{-1}||N| e^{-2\rho t}$,  where the last line has been obtained by applying to $R_t$ the same steps as $P_t$, and then using that   $\mathcal W_{K,2}$ is invariant by $Q$. For $t\geqslant 0$ large enough so that $\eta_t <1$, using that $(P_t)^* P_t$ is reversible with respect to $\mu_\infty$, we get by \cite[Proposition 30]{Ollivier}  that for all $f\in L^2(\mu_\infty)$ such that $\mu_\infty f= 0$,
 \[\| (P_t)^* P_t f\|_{L^2(\mu_\infty)} \ \leqslant \ \eta_t \|f\|_{L^2(\mu_\infty)},\]
 so that
 \[ \|   P_t f\|_{L^2(\mu_\infty)}^2 \ =\  \langle f, (P_t)^* P_t f\rangle \ \leqslant \  \eta_t \|f\|_{L^2(\mu_\infty)}^2\,. \]
Besides, if $\eta_t\geqslant 1$ then the inequality $\|P_t f\|_{L^2(\mu_\infty)} \leqslant \sqrt{\eta_t} \|  f\|_{L^2(\mu_\infty)}$ is obvious.
 \end{proof}

 \section*{Acknowledgments}
 
 This work has been supported by the French ANR projects EFI  (Entropy, flows, inequalities, ANR-17-CE40-0030) and SWIDIMS (ANR-20-CE40-0022).  Pierre Monmarché thanks the  Institut des Sciences du Calcul et des Données of Sorbonne Université for the creation of the MAESTRO team in relation with the study of the  generalised Langevin diffusions. He thanks Gabriel Stoltz for fruitfull discussions.

\bibliographystyle{plain}
\bibliography{biblio}

\begin{thebibliography}{10}

\bibitem{Arnold}
Franz Achleitner, Anton Arnold, and Dominik St\"{u}rzer.
\newblock Large-time behavior in non-symmetric {F}okker-{P}lanck equations.
\newblock {\em Riv. Math. Univ. Parma (N.S.)}, 6(1):1--68, 2015.

\bibitem{AlderWainwright}
B.~J. Alder and T.~E. Wainwright.
\newblock Studies in molecular dynamics. {I}. {G}eneral method.
\newblock {\em J. Chem. Phys.}, 31:459--466, 1959.

\bibitem{Bakry}
Dominique Bakry.
\newblock On {S}obolev and logarithmic {S}obolev inequalities for {M}arkov
  semigroups.
\newblock In {\em New trends in stochastic analysis ({C}haringworth, 1994)},
  pages 43--75. World Sci. Publ., River Edge, NJ, 1997.

\bibitem{BakryEmery}
Dominique Bakry and Michel \'{E}mery.
\newblock Hypercontractivit\'{e} de semi-groupes de diffusion.
\newblock {\em C. R. Acad. Sci. Paris S\'{e}r. I Math.}, 299(15):775--778,
  1984.

\bibitem{BakryGentilLedoux}
Dominique Bakry, Ivan Gentil, and Michel Ledoux.
\newblock {\em Analysis and geometry of {M}arkov diffusion operators}, volume
  348 of {\em Grundlehren der Mathematischen Wissenschaften [Fundamental
  Principles of Mathematical Sciences]}.
\newblock Springer, Cham, 2014.

\bibitem{Baudoin2}
Fabrice Baudoin.
\newblock {Wasserstein contraction properties for hypoelliptic diffusions}.
\newblock {\em arXiv e-prints}, page arXiv:1602.04177, February 2016.

\bibitem{Baudoin}
Fabrice Baudoin.
\newblock Bakry-\'{E}mery meet {V}illani.
\newblock {\em J. Funct. Anal.}, 273(7):2275--2291, 2017.

\bibitem{Schur}
Michele Benzi, Gene~H. Golub, and J\"{o}rg Liesen.
\newblock Numerical solution of saddle point problems.
\newblock {\em Acta Numer.}, 14:1--137, 2005.

\bibitem{BolleyGuillinMalrieu}
Fran{\c c}ois Bolley, Arnaud Guillin, and Florent Malrieu.
\newblock Trend to equilibrium and particle approximation for a weakly
  selfconsistent vlasov-fokker-planck equation.
\newblock {\em ESAIM: Mathematical Modelling and Numerical Analysis -
  Mod\'elisation Math\'ematique et Analyse Num\'erique}, 44(5):867--884, 2010.

\bibitem{Pavliotis}
Martin {Chak}, Nikolas {Kantas}, and Grigorios~A. {Pavliotis}.
\newblock {On the Generalised Langevin Equation for Simulated Annealing}.
\newblock {\em arXiv e-prints}, page arXiv:2003.06448, March 2020.

\bibitem{Chatterji1}
Xiang {Cheng}, Niladri~S. {Chatterji}, Peter~L. {Bartlett}, and Michael~I.
  {Jordan}.
\newblock {Underdamped Langevin MCMC: A non-asymptotic analysis}.
\newblock {\em arXiv e-prints}, page arXiv:1707.03663, July 2017.

\bibitem{Dalalyan}
Arnak~S. Dalalyan and Lionel Riou-Durand.
\newblock On sampling from a log-concave density using kinetic langevin
  diffusions.
\newblock {\em Bernoulli}, 26(3):1956--1988, 08 2020.

\bibitem{Paulin}
George {Deligiannidis}, Daniel {Paulin}, Alexandre {Bouchard-C{\^o}t{\'e}}, and
  Arnaud {Doucet}.
\newblock {Randomized Hamiltonian Monte Carlo as Scaling Limit of the Bouncy
  Particle Sampler and Dimension-Free Convergence Rates}.
\newblock {\em arXiv e-prints}, page arXiv:1808.04299, August 2018.

\bibitem{Eberle}
Andreas Eberle.
\newblock Reflection couplings and contraction rates for diffusions.
\newblock {\em Probab. Theory Related Fields}, 166(3-4):851--886, 2016.

\bibitem{EberleGuillinZimmer}
Andreas Eberle, Arnaud Guillin, and Raphael Zimmer.
\newblock Couplings and quantitative contraction rates for {L}angevin dynamics.
\newblock {\em Ann. Probab.}, 47(4):1982--2010, 2019.

\bibitem{Eckmann}
Jean-Pierre Eckmann, Claude-Alain Pillet, and Luc Rey-Bellet.
\newblock Non-equilibrium statistical mechanics of anharmonic chains coupled to
  two heat baths at different temperatures.
\newblock {\em Comm. Math. Phys.}, 201(3):657--697, 1999.

\bibitem{EthierKurtz}
Stewart~N. Ethier and Thomas~G. Kurtz.
\newblock {\em Markov processes}.
\newblock Wiley Series in Probability and Mathematical Statistics: Probability
  and Mathematical Statistics. John Wiley \& Sons, Inc., New York, 1986.
\newblock Characterization and convergence.

\bibitem{GeneralisedLangevin2}
Francesca Grogan, Huan Lei, Xiantao Li, and Nathan~A. Baker.
\newblock Data-driven molecular modeling with the generalized langevin
  equation.
\newblock {\em Journal of Computational Physics}, 418:109633, 2020.

\bibitem{MonmarcheGuillin}
Arnaud Guillin and Pierre Monmarch\'{e}.
\newblock Optimal linear drift for the speed of convergence of an hypoelliptic
  diffusion.
\newblock {\em Electron. Commun. Probab.}, 21:Paper No. 74, 14, 2016.

\bibitem{MonmarcheGuillin2020}
Arnaud {Guillin} and Pierre {Monmarch{\'e}}.
\newblock {Uniform long-time and propagation of chaos estimates for mean field
  kinetic particles in non-convex landscapes}.
\newblock {\em arXiv e-prints}, page arXiv:2003.00735, March 2020.

\bibitem{HairerHeatBath}
M.~Hairer.
\newblock How hot can a heat bath get?
\newblock {\em Comm. Math. Phys.}, 292(1):131--177, 2009.

\bibitem{Hormander}
Lars H\"{o}rmander.
\newblock Hypoelliptic second order differential equations.
\newblock {\em Acta Math.}, 119:147--171, 1967.

\bibitem{HOROWITZ}
Alan~M. Horowitz.
\newblock A generalized guided monte carlo algorithm.
\newblock {\em Physics Letters B}, 268(2):247 -- 252, 1991.

\bibitem{GeneralisedLangevin}
Gerhard Jung, Martin Hanke, and Friederike Schmid.
\newblock Generalized langevin dynamics: construction and numerical integration
  of non-markovian particle-based models.
\newblock {\em Soft Matter}, 14:9368--9382, 2018.

\bibitem{Kuwada1}
Kazumasa Kuwada.
\newblock Duality on gradient estimates and {W}asserstein controls.
\newblock {\em J. Funct. Anal.}, 258(11):3758--3774, 2010.

\bibitem{Kuwada2}
Kazumasa Kuwada.
\newblock Gradient estimate for {M}arkov kernels, {W}asserstein control and
  {H}opf-{L}ax formula.
\newblock In {\em Potential theory and its related fields}, RIMS
  K\^{o}ky\^{u}roku Bessatsu, B43, pages 61--80. Res. Inst. Math. Sci. (RIMS),
  Kyoto, 2013.

\bibitem{Ledoux}
Michel Ledoux.
\newblock L'alg\`ebre de lie des gradients it\'er\'es d'un g\'en\'erateur
  markovien - d\'eveloppements de moyennes et entropies.
\newblock {\em Annales scientifiques de l'\'Ecole Normale Sup\'erieure}, 4e
  s{\'e}rie, 28(4):435--460, 1995.

\bibitem{MangoubiSmith}
Oren {Mangoubi} and Aaron {Smith}.
\newblock {Rapid Mixing of Hamiltonian Monte Carlo on Strongly Log-Concave
  Distributions}.
\newblock {\em arXiv e-prints}, page arXiv:1708.07114, August 2017.

\bibitem{Menegaki}
Angeliki Menegaki.
\newblock Quantitative {R}ates of {C}onvergence to {N}on-equilibrium {S}teady
  {S}tate for a {W}eakly {A}nharmonic {C}hain of {O}scillators.
\newblock {\em J. Stat. Phys.}, 181(1):53--94, 2020.

\bibitem{MoiPDMP}
Pierre Monmarch\'{e}.
\newblock On {$\mathcal H^1$} and entropic convergence for contractive {PDMP}.
\newblock {\em Electron. J. Probab.}, 20:Paper No. 128, 30, 2015.

\bibitem{MoiGamma}
Pierre Monmarch\'{e}.
\newblock Generalized {$\Gamma$} calculus and application to interacting
  particles on a graph.
\newblock {\em Potential Anal.}, 50(3):439--466, 2019.

\bibitem{MoiLangevinCinetique}
Pierre {Monmarch{\'e}}.
\newblock {High-dimensional MCMC with a standard splitting scheme for the
  underdamped Langevin diffusion}.
\newblock {\em arXiv e-prints}, page arXiv:2007.05455, July 2020.

\bibitem{Mori}
Hazime Mori.
\newblock {Transport, Collective Motion, and Brownian Motion*)}.
\newblock {\em Progress of Theoretical Physics}, 33(3):423--455, 03 1965.

\bibitem{Neal}
Radford~M. Neal.
\newblock {MCMC} using {Hamiltonian} dynamics.
\newblock {\em Handbook of Markov Chain Monte Carlo}, 54:113--162, 2010.

\bibitem{Ollivier}
Yann Ollivier.
\newblock Ricci curvature of {M}arkov chains on metric spaces.
\newblock {\em J. Funct. Anal.}, 256(3):810--864, 2009.

\bibitem{Ottobre}
Michela Ottobre and Grigorios~A. Pavliotis.
\newblock Asymptotic analysis for the generalized {L}angevin equation.
\newblock {\em Nonlinearity}, 24(5):1629--1653, 2011.

\bibitem{PavliotisStoltz}
Grigorios~A. {Pavliotis}, Gabriel {Stoltz}, and Urbain {Vaes}.
\newblock {Scaling limits for the generalized Langevin equation}.
\newblock {\em arXiv e-prints}, page arXiv:2007.16087, July 2020.

\bibitem{ReyBellet}
Luc Rey-Bellet and Lawrence~E. Thomas.
\newblock Exponential convergence to non-equilibrium stationary states in
  classical statistical mechanics.
\newblock {\em Comm. Math. Phys.}, 225(2):305--329, 2002.

\bibitem{Tuckerman}
Mark~E. Tuckerman.
\newblock {\em Statistical mechanics: theory and molecular simulation}.
\newblock Oxford Graduate Texts. Oxford University Press, Oxford, 2010.

\bibitem{Villani2009}
C.~Villani.
\newblock Hypocoercivity.
\newblock {\em Mem. Amer. Math. Soc.}, 202(950):iv+141, 2009.

\bibitem{Villani}
C{\'e}dric Villani.
\newblock {\em Optimal transport, old and new}, volume 338 of {\em Grundlehren
  der Mathematischen Wissenschaften [Fundamental Principles of Mathematical
  Sciences]}.
\newblock Springer-Verlag, Berlin, 2009.

\bibitem{VonRenesseSturm}
Max-K. von Renesse and Karl-Theodor Sturm.
\newblock Transport inequalities, gradient estimates, entropy and ricci
  curvature.
\newblock {\em Communications on Pure and Applied Mathematics}, 58(7):923--940,
  2005.

\bibitem{wang2006functional}
Fengyu Wang.
\newblock {\em Functional inequalities Markov semigroups and spectral theory}.
\newblock Elsevier, 2006.

\bibitem{Zwanzig}
Robert Zwanzig.
\newblock Memory effects in irreversible thermodynamics.
\newblock {\em Phys. Rev.}, 124:983--992, Nov 1961.

\end{thebibliography}

\end{document}